\documentclass[12pt]{article}
\input epsf.tex

\usepackage{graphicx}
\usepackage{amsmath,amsthm,amsfonts,amscd,amssymb,comment,eucal,latexsym,mathrsfs}
\usepackage{stmaryrd}
\usepackage{bbm}
\usepackage[all]{xy}

\usepackage{epsfig}

\usepackage[all]{xy}
\xyoption{poly}
\usepackage{fancyhdr}
\usepackage{wrapfig}
\usepackage{epsfig}

\usepackage{hyperref}

\theoremstyle{plain}
\newtheorem{thm}{Theorem}[section]

\newtheorem{prop}[thm]{Proposition}
\newtheorem{lem}[thm]{Lemma}
\newtheorem{cor}[thm]{Corollary}

\newtheorem{claim}[thm]{Claim}

\theoremstyle{definition}

\theoremstyle{remark}
\newtheorem{remark}{Remark}

\advance \topmargin by -\headheight
\advance \topmargin by -\headsep
\evensidemargin \oddsidemargin
\def\A{{\cal A}}

\def\cA{{\cal A}}
\def\cS{{\cal S}}
\def\cT{{\cal T}}

 \def\bfb{{\bf{b}}}                        \def\bfz{{\bf{z}}}

\def\bB{{\bf B}}

\def\cc{{\curvearrowright}}

\def\d{{\cal d}}

\def\H{{\mathbb H}}

\def\cK{{\mathcal K}}

\def\cM{{\cal M}}

\def\N{{\mathbb N}}

\def\bP{{\mathbb P}}
\def\cP{{\mathcal P}}
\def\cQ{{\mathcal Q}}
\def\bcP{{\overline {\mathcal P}}}

\def\R{{\mathbb R}}

\def\cR{\mathcal R}

\def\T{{\mathbb T}}
\def\X{{\mathbb X}}

\def\chix{{\raise.5ex\hbox{$\chi$}}}

\def\Z{{\mathbb Z}}

\def\cmu{\check{\mu}}
\def\Irr{{\textrm{Irr}}}

\def\Pol{{Z}}

\def\slashslash{\backslash\kern-.3em\backslash}
\def\ind{\mathbbm{1}}

\newcommand\Aut{\operatorname{Aut}}

\newcommand\bnd{\operatorname{bnd}}
\newcommand\GOOD{\operatorname{GOOD}}

\newcommand\Fix{{\operatorname{Fix}}}
\newcommand\Hom{\operatorname{Hom}}

\renewcommand\Im{\operatorname{Im}}

\newcommand\Isom{\operatorname{Isom}}
\newcommand\Conj{\operatorname{Conj}}

\newcommand\MALG{\operatorname{MALG}}
\newcommand\Map{{\operatorname{Map}}}

\newcommand\Part{\operatorname{Part}}

\newcommand\Proj{\operatorname{Proj}}

\newcommand\res{\upharpoonright}

\newcommand\Fact{{\operatorname{Fact}}}

\newcommand\Sub{\operatorname{Sub}}

\newcommand\stab{\operatorname{stab}}

\newcommand{\csuchthat}{\, :\,}

\begin{document}
\title{Generic Stationary Measures and Actions}
\author{Lewis Bowen\footnote{University of Texas at
    Austin. Supported in part by NSF grant DMS-0968762, NSF CAREER
    Award DMS-0954606 and BSF grant 2008274.}, Yair
  Hartman\footnote{Weizmann Institute of Science. Supported by the
    European Research Council, grant 239885.} and Omer
  Tamuz\footnote{California Institute of Technology.}}
\maketitle

\begin{abstract}
  Let $G$ be a countably infinite group, and let $\mu$ be a generating
  probability measure on $G$. We study the space of $\mu$-stationary
  Borel probability measures on a topological $G$ space, and in
  particular on $Z^G$, where $Z$ is any perfect Polish space. We also
  study the space of $\mu$-stationary, measurable $G$-actions on a
  standard, nonatomic probability space.

  Equip the space of stationary measures with the weak* topology. When
  $\mu$ has finite entropy, we show that a generic measure is an
  essentially free extension of the Poisson boundary of $(G,\mu)$.
  When $Z$ is compact, this implies that the simplex of
  $\mu$-stationary measures on $Z^G$ is a Poulsen simplex. We show
  that this is also the case for the simplex of stationary measures on
  $\{0,1\}^G$.
   
  We furthermore show that if the action of $G$ on its Poisson
  boundary is essentially free then a generic measure is isomorphic to
  the Poisson boundary.
 
  Next, we consider the space of stationary actions, equipped with a
  standard topology known as the weak topology. Here we show that when
  $G$ has property (T), the ergodic actions are meager. We also
  construct a group $G$ without property (T) such that the ergodic
  actions are not dense, for some $\mu$.

  Finally, for a weaker topology on the set of actions, which we call
  the very weak topology, we show that a dynamical property (e.g.,
  ergodicity) is topologically generic if and only if it is generic in
  the space of measures. There we also show a Glasner-King type 0-1
  law stating that every dynamical property is either meager or
  residual.
\end{abstract}


\noindent
\tableofcontents

\section{Introduction}
This paper is motivated by two subjects: {\em genericity in dynamics}
and {\em stationary actions}. We begin our introduction with
background on genericity in dynamics.

\subsection{Genericity}

There is a long history of topological genericity in dynamics
beginning with Halmos~\cite{halmos1944general}, who showed that a
generic automorphism is weakly mixing. More precisely, he studied the
group $\Aut(X,\nu)$ of measure-preserving transformations of a
standard Lebesgue probability space $(X,\nu)$ (in which two
transformations that agree modulo measure zero are identified). This
group is naturally equipped with a Polish topology. Halmos proved that
the set of weakly mixing transformations is a dense $G_\delta$ subset
of $\Aut(X,\nu)$. Since then, it has been proven that a generic
measure-preserving transformation has zero entropy, rank
one~\cite{ferenczi-rank}, is rigid, has interesting spectral
properties~\cite{junco-generic, stepin-generic} and so on. There also
many interesting results about generic
homeomorphisms~\cite{hochman-generic}.

Instead of studying the space of transformations, one may study the
space of measures. Precisely, consider a homeomorphism $T: X\to X$ of
a topological space $X$. Let $\cP_T(X)$ denote the space of all
$T$-invariant Borel probability measures on $X$ with the weak*
topology. If $X$ is compact, then $\cP_T(X)$ is a Choquet simplex: it
is a convex, compact subset of a locally convex vector space, such
that every measure in $\cP_T(X)$ has a unique representation as the
barycenter of a probability measure on the set of extreme points of
$\cP_T(X)$ (and the extreme points are exactly the ergodic measures
for $T$). Given any abstract Choquet simplex $\Sigma$, there exists a
compact metric space $X$ and a homeomorphism $T:X\to X$ such that
$\cP_T(X)$ is affinely homeomorphic to
$\Sigma$~\cite{down-serafin}. Therefore, it is natural to look for
``special'' homeomorphisms. Indeed, there is a canonical choice:
consider the product space $[0,1]^\Z$ with the shift action
$T:[0,1]^\Z \to [0,1]^\Z$ given by $(Tx)_i=x_{i+1}$. It follows from
Rohlin's Lemma (for example) that the ergodic measures are dense (and
therefore residual) in
$\cP_T([0,1]^\Z)$. By~\cite{lindenstrauss1978poulsen} there is a
unique Choquet simplex with this property (up to affine
homeomorphisms) called the {\em Poulsen simplex}.

This result was greatly generalized in an influential paper of Glasner
and Weiss~\cite{glasner1997kazhdan}. They considered an arbitrary
countable group $G$ acting by homeomorphisms on a compact metrizable
space $X$ (they also considered locally compact groups but we will not
need that here). Let $\cP_G(X)$ denote the space of $G$-invariant
Borel probability measures on $X$ with the weak* topology. As before
this is a Choquet simplex and the ergodic measures are the extreme
points. When $G$ has property (T), they showed that the set of ergodic
measures is closed in $\cP_G(X)$, and therefore $\cP_G(X)$ is a {\em
  Bauer simplex}. When $G$ does not have property (T) then they show
that $\cP_G(\{0,1\}^G)$ is a Poulsen simplex, and $G$ acts on
$\{0,1\}^G$ by $(gx)_k = x_{g^{-1}k}$ for $x\in \{0,1\}^G, g,k\in
G$. Their proof extends to any $\cP_G(W^G)$, with $W$ a
non-trivial compact space.

The fact that ergodic transformations are residual in $\Aut(X,\nu)$
and ergodic measures are residual in $\cP_T([0,1]^\Z)$ is no accident:
Glasner and King proved that for any dynamical property P, a generic
element of $\Aut(X,\nu)$ has P if and only if a generic measure in
$\cP_T([0,1]^\Z)$ has P~\cite{GK98}. In fact, their proof extends to
measure preserving actions of all countable groups.

To make a more precise statement, let $G$ denote a countable
group. The space $A(G,X,\nu)=\Hom(G,\Aut(X,\nu))$ of all homomorphisms
from $G$ to $\Aut(X,\nu)$ is equipped with the topology of pointwise
convergence, under which it is a Polish space.  Glasner and King
proved that if P is any dynamical property then a generic action in
$A(G,X,\nu)$ has P if and only if a generic measure in
$\cP_G([0,1]^G)$ has P. The precise statement is recounted in
\S~\ref{sec:correspondence} of this paper.

As noted in~\cite{glasner2006every}, these results imply the following
dichotomy: if $G$ has (T) then the ergodic actions form a meager
subset of $A(G,X,\nu)$ while if $G$ does not have (T) then the ergodic
actions form a residual subset of $A(G,X,\nu)$. This result was
extended by Kerr and Pichot~\cite{kerr2008asymptotic} to show that if
$G$ does not have (T) then the weakly mixing actions are a dense
$G_\delta$ subset. Kerr and Pichot also generalize this result to
$C^*$-dynamical systems.

Let us also mention here the {\em weak Rohlin property} as well as the
{\em 0-1 law} of Glasner and King~\cite{GK98}.  $\Aut(X,\nu)$ acts
continuously on this space by conjugation, and a group has the weak
Rohlin property if $A(G,X,\nu)$ has a dense
$\Aut(X,\nu)$-orbit. By~\cite{glasner2006every} all countable groups
have the weak Rohlin property.  Any group with the weak Rohlin
property (and hence any countable group) obeys a 0-1 law: every
Baire-measurable dynamical property of $A(G,X,\nu)$ is either
residual or meager.

\subsection{Stationarity}

Let $G$ be a countable group and $\mu$ a probability measure on $G$
whose support generates $G$ as a semigroup. An action $G \cc (X,\nu)$
on a probability space is {\em $\mu$-stationary} (or just {\em
  stationary} if $\mu$ is understood) if
\begin{align}
  \label{eq:stationarity}
  \sum_{g \in G} \mu(g)  g_*\nu =\nu.
\end{align}
In this case, $\nu$ is said to be {\em $\mu$-stationary}, and it
follows that the action is nonsingular. Stationary actions are
intimately related to random walks and harmonic functions on groups,
as well as to the Poisson boundary~\cite{furstenberg1963noncommuting,
  furstenberg1971random, furstenberg1974boundary}, which is itself a
stationary space.

In principle stationary actions exist in abundance: if $G$ acts
continuously on a compact metrizable space $X$ then there exists a
$\mu$-stationary probability measure on $X$. By contrast, if $G$ is
non-amenable then an invariant measure need not exist. However, there
are surprisingly few explicit constructions of stationary actions:
aside from measure-preserving actions and Poisson boundaries, there
are constructions from invariant random
subgroups~\cite{bowen2010random, hartman2012furstenberg} and methods
for combining stationary actions via
joinings~\cite{furstenberg2009stationary}. There is a general
structure theory of stationary
actions~\cite{furstenberg2009stationary} and a very deep structure
theory in the case that $G$ is a higher rank semisimple Lie
group~\cite{nevo2000rigidity, nevo-zimmer-annals}. There is also a
growing literature on classifying stationary
actions~\cite{bourgain2007invariant, bourgain2011stationary,
  benoist2011mesures} and stationarity has found recent use in proving
nonexistence of $\sigma$-compact topological
models~\cite{conley-stationary} for certain Borel actions.

The Poisson boundary of $(G,\mu)$ is the space of ergodic components
of the shift action on $(G^\N,\bP_\mu)$ where $\bP_\mu$ is the law of the
random walk on $G$ with
$\mu$-increments~\cite{kaimanovich1983random}. We denote the Poisson
boundary by $\Pi(G,\mu)$. This space was introduced by Furstenberg who
showed that the space of bounded $\mu$-harmonic functions on $G$ is
naturally isomorphic with
$L^\infty(\Pi(G,\mu))$~\cite{furstenberg1963noncommuting,
  furstenberg1971random}. It plays a central role in the structure
theory of stationary actions~\cite{furstenberg2009stationary} and is
important in rigidity theory~\cite{bader2006factor, nevo2000rigidity,
  hartman2013stabilizer}. It also plays a key role in this paper, and
so we define it formally in \S~\ref{sec:Poisson}.

\subsection{Main results}

Motivated by the above genericity results we ask the following
questions: is a generic stationary action ergodic? Is a generic
stationary measure ergodic?  Does it depend on whether $G$ has
property (T)? Is there a Glasner-King type correspondence principle
relating dynamical properties of generic stationary actions and
measures? Is there a Glasner-King 0-1 law for stationary actions?  We
answer some of these questions next.%

\subsubsection{Spaces of measures}
Let $G$ be a discrete, countable group. Our investigations begin with
spaces of measures. So if $G$ acts by homeomorphisms on a topological
space $Y$, let $\cP_\mu(Y)$ denote the space of all $\mu$-stationary
Borel probability measures on $Y$ with the weak*
topology. By~\eqref{eq:stationarity} this is a closed subspace of
$\cP(Y)$, the space of all Borel probability measures on $Y$. Its
extreme points are the ergodic measures and it is a Choquet simplex if
$Y$ is compact and Hausdorff~\cite{bader2006factor}. Our first result
is:

\begin{thm}[Generic stationary measures]
  \label{thm:generic-poisson-boundary}
  Let $\Pol$ be a perfect Polish space, and let $\mu$ have finite
  entropy. Then a generic measure in $\cP_\mu(\Pol^G)$ is an ergodic,
  essentially free extension of the Poisson boundary, denoted
  $\Pi(G,\mu)$. Moreover, if the action $G \curvearrowright
  \Pi(G,\mu)$ is essentially free, then a generic measure $\nu \in
  \cP_\mu(\Pol^G)$ is such that $G \cc (\Pol^G,\nu)$ is measurably
  conjugate to $G \cc \Pi(G,\mu)$.
\end{thm}
Here, $G \cc (B,\nu)$ is an extension of the Poisson boundary
$\Pi(G,\mu)$ if there exists a $G$-equivariant factor $(B,\nu) \to
\Pi(G,\mu)$.  Note that an ergodic, essentially free extension of the
Poisson boundary always exists; one can take, for example, the product
of the Poisson boundary with a Bernoulli shift.

\begin{cor}
  If $\Pol$ is a compact perfect Polish space then $\cP_\mu(\Pol^G)$
  is a Poulsen simplex.
\end{cor}

Observe that we do not put any conditions on the group $G$ in the
above results. In particular, $G$ is allowed to have property
(T). Perhaps the most unusual aspect of the result above occurs when
$G$ acts essentially freely on its Poisson boundary. For in this case,
there is a generic measure-conjugacy class. This might be considered
analogous to the Kechris-Rosendal result that there is a generic conjugacy
class in the group of homeomorphisms of the Cantor
set~\cite{kechris-rosendal-2007, akin-weiss}. See also~\cite{glasner2008topological} for
other examples of transformation groups with generic conjugacy
classes.

The main technical component of the proof of
Theorem~\ref{thm:generic-poisson-boundary} is
Theorem~\ref{thm:0-1-dense}, which states that given an ergodic,
essentially free extension of the Poisson boundary $G\cc (B,\nu)$ and
a compact metric space $W$, the set of measures on $\cP_\mu(W^G)$ that
are $G$-factors of $G\cc (B,\nu)$ is dense.

To motivate our interest on groups that act freely on their Poisson
boundaries, we provide the following straightforward claim.
\begin{prop}
  \label{prop:free-boundary}
  Let $G$ be a torsion-free, non-elementary word hyperbolic
  group. Then $G$ acts essentially freely on its Poisson boundary
  $\Pi(G,\mu)$, for any generating measure $\mu$.
\end{prop}

\subsubsection{Spaces of actions}

Next we turn our attention to spaces of actions. Here, it appears that
there are two natural choices for the topology on the space of
stationary actions. To be precise, let $\Aut^*(X,\nu)$ denote the
group of nonsingular transformations of $(X,\nu)$ in which two such
transformations are identified if they agree up to null sets. We embed
this group into $\Isom(L^p(X,\nu))$ via
$$T \mapsto U_{T,p}, \quad U_{T,p}(f) = \left( \frac{dT_*\nu}{d\nu}(x)\right)^{1/p} f\circ T^{-1}.$$
We equip $\Isom(L^p(X,\nu))$ with either the weak or strong operator
topology and $\Aut^*(X,\nu)$ with the subspace topology. From results
in~\cite{megrelishvili2001operator, CK79}, it follows that only two
different topologies on $\Aut^*(X,\nu)$ result from this construction:
the topology derived from the weak operator topology on
$\Isom(L^1(X,\nu))$ and the topology derived from any other choice of
$1\le p <\infty$ and (weak/strong). The latter topology has been
studied previously~\cite{ionescu-tulcea, CK79} and is called the {\em
  weak topology}. Therefore, we call the topology derived from the
weak operator topology on $\Isom(L^1(X,\nu))$, the {\em very weak
  topology}. Both of these topologies are Polish topologies. However,
only the weak topology is a group topology.

Next we let $A^*(G,X,\nu)=\Hom(G,\Aut^*(X,\nu))$ be the space of
homomorphisms of $G$ into $\Aut^*(X,\nu)$ with the topology of
pointwise convergence and $A_\mu(G,X,\nu) \subset A^*(G,X,\nu)$ the
subspace of $\mu$-stationary actions with the subspace topology. This
gives two distinct topologies on $A_\mu(G,X,\nu)$ (depending on the
choice of topology on $\Aut^*(X,\nu)$) which we also call the weak and
very weak topologies. Both topologies are Polish and both topologies
restrict to the same topology on $A(G,X,\nu)$ (which is the usual one,
as studied in~\cite{Kechris-global-aspects, kerr2008asymptotic} for
example). Note that as in the case of the measure preserving actions,
the group $\Aut(X,\nu)$ acts on $A_\mu(G,X,\nu)$ by conjugations. This
action is continuous under both topologies on $A_\mu(G,X,\nu)$.

The weak topology on $A_\mu(G,X,\nu)$ is perhaps more natural, since
it is derived from the group topology on $\Isom(L^1(X,\nu))$. However,
under the very weak topology, $A_\mu(G,X,\nu)$ better resembles the
space of measure preserving actions $A(G,X,\nu)$. Indeed, as we will
show below, under the very weak topology there is always a dense
$\Aut(G,X,\nu)$-orbit in $A_\mu(G,X,\nu)$, and hence a 0-1 law. This
is not true in the weak topology, unless the only stationary measures
are invariant.

\subsubsection{The weak topology}

We will prove:
\begin{thm}
  \label{thm:T-erg-meager}
  If $G$ has property (T) then the set of ergodic actions in
  $A_\mu(G,X,\nu)$ is nowhere dense, when $A_\mu(G,X,\nu)$ is endowed
  with the weak topology.
\end{thm}
Recall that the same result holds in the measure-preserving
case~\cite{glasner1997kazhdan}. Therefore, it makes sense to ask: if
$G$ does not have (T) then are the ergodic actions generic? In this
generality, the answer is no: we provide an explicit counterexample.

\begin{thm}
  \label{thm:non-T-non-dense}
  There exists a countable group $G$ that does not have property (T), and a
  generating probability measure $\mu$ on $G$ such that the set of
  ergodic measures is not dense in $A_\mu(G,X,\nu)$, when
  $A_\mu(G,X,\nu)$ is endowed with the weak topology.
\end{thm}

At this point, it is natural to ask whether, because the ergodic
measures in the example above are not dense, if they must be
meager. However there is no dense $\Aut(X,\nu)$-orbit, and no
Glasner-King type 0-1 law:
\begin{prop}\label{prop:no-01-law}
  If $(G,\mu)$ has a nontrivial Poisson boundary then, under the weak
  topology on $A_\mu(G,X,\nu)$, there does not exist in
  $A_\mu(G,X,\nu)$ a dense $\Aut(X,\nu)$-orbit, and there does exist
  an $\Aut(X,\nu)$-invariant Borel subset that is neither meager
  nor residual.
\end{prop}

On the other hand, if the Poisson boundary of $(G,\mu)$ is trivial then all
stationary actions are measure-preserving, and there is a dense orbit
and a 0-1 law. In this case, the group $G$ is necessarily amenable. Incidentally, for amenable groups, it
is an open question whether the ergodic actions are dense in $A_\mu(G,X,\nu)$ with respect to the weak topology. 

\subsubsection{The very weak topology}


We prove that the Glasner-King correspondence principle generalizes to
stationary actions with respect to the very weak topology:
\begin{thm}[Correspondence principle]
  \label{thm:generic-equivalence}
  Let $\Pol$ be a perfect Polish space, and let $P$ be a dynamical
  property.  A generic action in $A_\mu(G,X,\nu)$ has $P$ iff a
  generic measure in $\cP_\mu(\Pol^G)$ has $P$, when $A_\mu(G,X,\nu)$
  is endowed with the very weak topology.
\end{thm}
It follows from Theorems~\ref{thm:generic-poisson-boundary}
and~\ref{thm:generic-equivalence} that under this topology, a generic
action is an essentially free extension of the Poisson boundary, and
is isomorphic to the Poisson boundary when the action on it is
essentially free. Another interesting consequence of
Theorem~\ref{thm:generic-equivalence} is that if $\Pol_1$ and $\Pol_2$
are perfect Polish spaces, then a generic measure in $\cP(\Pol_1^G)$
has a dynamical property iff a generic measure in $\cP_\mu(\Pol_2^G)$
has this property. We use this in the proof of
Theorem~\ref{thm:generic-poisson-boundary}. 

We prove that under the very weak topology, $A_\mu(G,X,\nu)$ does have
a dense $\Aut(X,\nu)$-orbit.
\begin{thm}\label{thm:weak-rohlin}
  For any discrete group $G$ with a generating measure $\mu$, there
  exists in $A_\mu(G,X,\nu)$ a dense $\Aut(X,\nu)$-orbit, with respect
  to the very weak topology.
\end{thm}

A consequence is a Glasner-King type 0-1 law.
\begin{cor}[0-1 law for stationary actions]\label{cor:0-1law}
  Every $\Aut(X,\nu)$-invariant Baire measurable subset of
  $A_\mu(G,X,\nu)$ is either meager or residual in the very weak
  topology.
\end{cor}

{\bf Acknowledgments}. We are grateful to Ita\"\i~Ben Yaacov, Julien
Melleray and Todor Tsankov for helping us understand the different
topologies on the group $\Aut^*(X,\lambda)$ of nonsingular
transformations of a Lebesgue space. Part of this paper was written
while all three authors attended the trimester program ``Random Walks
and Asymptotic Geometry of Groups'' at the Henri Poincar\'e Institute
in Paris. We are grateful to the Institute for its support.

\section{Definitions and preliminaries}
\label{sec:prelim}
\subsection{Nonsingular and stationary measures}
Let $G \curvearrowright Y$ be a continuous action of a countable group
on a compact Hausdorff space. A particularly interesting case is when
$Y = W^G$ for some compact metric space $W$ with cardinality $>1$, the
topology is the product topology, and $G$ acts by left translations.
We denote by $\cP(Y)$ the space of Borel probability measures on $Y$,
equipped with the weak* topology. This is a compact Polish space.

A subspace of $\cP(Y)$ is $\cP^*_G(Y)$, the space of $G$
quasi-invariant measures on $Y$. Those are the measures $\nu \in
\cP(Y)$ such that $g_*\nu$ and $\nu$ are equivalent - that is,
mutually absolutely continuous - for all $g \in G$.

A probability measure $\mu$ on $G$ is said to be {\em generating} if
its support generates $G$ as a semigroup.  Given such a measure $\mu$,
a subspace of $\cP^*_G(Y)$ is the closed set of $\mu$-stationary
measures $\cP_\mu(Y)$. Those are the measures that
satisfy~\eqref{eq:stationarity}. This is also a Polish space.

Finally, $\cP_G(Y) \subseteq \cP_\mu(Y)$ is the space of $G$-invariant
measures. This series of inclusions is summarized as follows:
\begin{align*}
  \cP_G(Y) \, \subseteq \, \cP_\mu(Y) \, \subseteq \, \cP_G^*(Y) \, \subseteq \, \cP(Y).
\end{align*}

\subsubsection{The Poisson boundary}\label{sec:Poisson}
The Poisson boundary $\Pi(G,\mu)$ is an important measurable
$\mu$-stationary action on an abstract probability space. It was
introduced by Furstenberg~\cite{furstenberg1971random} in the context
of Lie groups, or, more generally, locally compact second countable
groups; we will define it for countable groups.

So let $G$ be a countable group and $\mu$ a probability measure on $G$
whose support generates $G$ as a semigroup. Let $\bP_\mu$ be the
push-forward of $\mu^\N$ under the map $(g_1,g_2,g_3,\ldots) \mapsto
(g_1,g_1g_2,g_1g_2g_3,\ldots)$. The space $(G^\N,\bP_\mu)$ is the
space of random walks on $G$ with $\mu$-increments.

Consider the natural shift action on $G^\N$ given by
$(g_1,g_2,g_3,\ldots) \mapsto (g_2,g_3,\ldots)$. The Poisson boundary
$\Pi(G,\mu)$ is the Mackey point realization~\cite{mackey1962point} of
the shift-invariant sigma-algebra of $(G^\N,\bP_\mu)$, and can be
thought of as the set of possible asymptotic behaviors of the random
walk. We refer the reader to Furman~\cite{furman2002random} for an
in-depth discussion.

An important property of the Poisson boundary is that the $G$-action
on it is amenable, in Zimmer's sense~\cite{zimmer1978amenable}. This
fact is an important ingredient in the proof of
Theorem~\ref{thm:generic-poisson-boundary}, and we use it in the proof
of Lemma~\ref{lem:rohlin}. Another important property, which we
discuss in the next section, is that the Furstenberg entropy of the
Poisson boundary is maximal.

Let $G \cc (B,\beta)$ be the Poisson boundary of $(G,\mu)$. We call
measure $\nu \in \cP_\mu(Y)$ {\em Poisson} if $G \cc (Y,\nu)$ is
measurably conjugate to $G \cc (B,\beta)$. Let $\cP^{Poisson}_\mu(Y)
\subset \cP_\mu(Y)$ denote the subset of Poisson measures. A measure
$\nu \in \cP_\mu(Y)$ is an extension of the Poisson boundary if there
exists a $G$-equivariant factor $\pi \colon Y \to B$ such that
$\pi_*\nu = \beta$.

\subsubsection{Furstenberg entropy}

The {\em Furstenberg entropy}~\cite{furstenberg1963noncommuting} of a
$\mu$-stationary measure $\nu \in \cP_\mu(Y)$ is given by
\begin{align*}
  h_\mu(Y,\nu) = \sum_{g \in G}\mu(g)\int_Y-\log\frac{d\nu}{dg_*\nu}(y)dg_*\nu(y).
\end{align*}
We also refer to $h_\mu(\cdot)$ as {\em $\mu$-entropy}.

Furstenberg entropy is an important measure-conjugacy invariant of
stationary actions; for example, when the Shannon entropy of $\mu$ is
finite, then the only proximal stationary space (i.e., a factor of the
Poisson boundary) with maximal Furstenberg entropy is the Poisson
boundary~\cite{kaimanovich1983random}. In general (i.e., even when the
entropy of $\mu$ is infinite), every stationary space has Furstenberg
entropy that is at most that of the Poisson boundary, and the latter
is bounded by the Shannon entropy of $\mu$. Because of this fact we
say that a stationary action has {\em maximum $\mu$-entropy} if its
$\mu$-entropy equals the $\mu$-entropy of the Poisson boundary.

\section{$G_\delta$ subsets of the space of measures} 
Let $Y$ be a compact metric space on which $G$ acts by homeomorphisms.
Recall that a measure $\nu \in \cP_\mu(Y)$ is
\begin{itemize}
\item {\em ergodic} if for every $G$-invariant measurable subset $E \subset Y$, $\nu(E) \in \{0,1\}$,
\item {\em maximal} if the $\mu$-entropy of $G \cc (Y,\nu)$ is the same as the $\mu$-entropy of $G$ acting on the Poisson boundary,
\item {\em proximal} if $G \cc (Y,\nu)$ is a measurable factor of the Poisson boundary action $G \cc \Pi(G,\mu)$,
\item {\em Poisson} if $G \cc (Y,\nu)$ is measurably-conjugate to
  the Poisson boundary action $G \cc \Pi(G,\mu)$,
\item {\em essentially free} if for each $g \in G$, the set of $G$
  fixed points $\{y \in Y \,:\, gy=y\}$ has $\nu$-measure zero.
\end{itemize}

The purpose of this section is to prove:

\begin{thm}
  \label{thm:G-deltas}
  Let 
  \begin{itemize}
  \item $\cP^e_\mu(Y)\subset \cP_\mu(Y)$ denote the subset of ergodic measures,
  \item $\cP^{max}_\mu(Y) \subset \cP_\mu(Y)$ denote the subset of
    maximum $\mu$-entropy measures,
  \item $\cP^{proximal}_\mu(Y) \subset \cP_\mu(Y)$ denote the subset
    of proximal measures,
  \item $\cP^{free}_\mu(Y) \subset \cP_\mu(Y)$ denote the subset of
    essentially free measures,
  \item $\cP^{Poisson}_\mu(Y) \subset \cP_\mu(Y)$ denote the subset of
    Poisson measures.
  \end{itemize}
  Then $\cP^e_\mu(Y),\cP^{max}_\mu(Y)$, $\cP^{proximal}_\mu(Y)$ and
  $\cP^{free}_\mu(Y)$ are $G_\delta$ subsets of $\cP_\mu(Y)$.

  If the Shannon entropy $H(\mu)<\infty$ then $\cP^{Poisson}_\mu(Y)$
  is also a $G_\delta$-subset of $\cP_\mu(Y)$.
\end{thm}


To get started, we prove that $\cP^e_\mu(Y)$, $\cP^{max}_\mu(Y)$ and
$\cP^{free}_\mu(Y)$ are $G_\delta$ subsets after the next (standard)
lemma.
\begin{lem}
  \label{thm:ex-G-delta}
  Let $\Delta$ be a Choquet simplex and $\Delta^e \subset \Delta$
  denote the subset of extreme points. Then $\Delta^e$ is a $G_\delta$
  subset of $\Delta$.
\end{lem}
\begin{proof}
  Let $d$ denote a continuous metric on $\Delta$. For each integer
  $n>0$ let $\Delta_n$ denote the set of all $x \in \Delta$ such that
  there exists $y,z \in \Delta$ with $d(y,z) \ge 1/n$ such that $x =
  (1/2)(y+z)$. Then $\Delta_n$ is closed in $\Delta$ and $\Delta^e =
  \cap_{n=1}^\infty \Delta\setminus \Delta_n$.
\end{proof}

The proof of the following corollary is straightforward, and involves
the application of known results from the measure-preserving case to
the stationary case.
\begin{cor}\label{cor:ergodic}
  $\cP^e_\mu(Y)$, $\cP^{max}_\mu(Y)$ and $\cP^{free}_\mu(Y)$ are $G_\delta$ subsets. 
\end{cor}
\begin{proof}
  By the ergodic decomposition theorem for stationary measures,
  $\cP^e_\mu(Y)$ is the set of extreme points of
  $\cP_\mu(Y)$~\cite{bader2006factor}. So the previous lemma implies
  $\cP^e_\mu(Y)$ is a $G_\delta$. By~\cite[Theorem 1]{Po75}, the map
  $\nu \mapsto h_\mu(Y,\nu)$ is lower semi-continuous on
  $\cP_\mu(Y)$. This implies $\cP^{max}_\mu(Y)$ is a $G_\delta$. To
  see that $\cP^{free}_\mu(Y)$ is a $G$-delta, note that for each $g
  \in G$, the set of $g$ fixed points $F_g = \{y \in Y\,:\,gy=y\}$ is
  closed, by the continuity of the $G$-action. So the portmanteau
  Theorem implies the set $M_{g,n} = \{\nu \in \cP_\mu(Y) \,:\,
  \nu(F_g) < 1/n\}$ is open for all $n>1$. Since $\cP^{free}_\mu(Y) =
  \cap_{g \in G\setminus \{e\}} \cap_{n=1}^\infty M_g$, it is a
  $G_\delta$.
\end{proof}

\subsection{$\Z$-invariant measures from stationary measures}
\label{invariant_stationary_measures}

In order to prove that proximal measures form a $G_\delta$ subset of
$\cP_\mu(Y)$, we obtain an affine homeomorphism between $\cP_\mu(Y)$
and a certain space of $\Z$-invariant measures. This idea is inspired
by \cite{furstenberg2009stationary}, and parts of what follows appear
in~\cite{furman2002random} (see, e.g., Proposition 1.3 there, as well
as section 2.3). We never-the-less provide complete proofs, for the
reader's convenience.

Given a measure $\mu$ on $G$, let $\cmu$ be the measure on $G$ given
by $\cmu(A) = \mu(\{g \in G\,:\,g^{-1} \in A\})$. To begin, we let $G$
have the discrete topology, $G^\Z$ the product topology and $\cP(G^\Z
\times Y)$ the weak* topology. Let $\cP(G^\Z\times Y | \cmu^\Z)$
denote the set of all measures $\lambda \in \cP(G^\Z \times Y)$ whose
projection to the first coordinate is $\cmu^\Z$. We view $\cP(G^\Z \times Y| \cmu^\Z)$ as a subspace of $\cP(G^\Z \times Y)$ with the subspace topology. In Appendix \ref{sec:weak} we show that this topology on $\cP(G^\Z \times Y| \cmu^\Z)$ is independent of the choice of topology on $G^\Z$.

We will show that $\cP_\mu(Y)$ is affinely homeomorphic with a
subspace of $\cP(G^\Z\times Y | \cmu^\Z)$. To define this subspace,
let $r:\Z \times G^\Z \to G$ be the random walk cocycle:
\begin{displaymath}
  r(n,\omega)= \left\{ \begin{array}{cc}
      (\omega_1\cdots \omega_n)^{-1} & n \ge 1 \\
      1_G & n = 0 \\ 
      \omega_{n+1}\cdots \omega_0 & n <0
    \end{array}\right.
\end{displaymath}
Note that $r$ satisfies the cocycle equation
$$r(n+m,\omega)=r(n,\sigma^m\omega)r(m,\omega)$$
where $\sigma$ is the left shift-operator on $G^\Z$ defined by
$\sigma(\omega)_i = \omega_{i+1}$ for $i\in \Z$. 

Define the transformation $T:G^\Z \times Y \to G^\Z \times Y$ by $T(\omega,y)=(\sigma\omega, \omega_1^{-1}y)$. Observe that for any $n\in \Z$,
$$T^n(\omega, y) = (\sigma^n\omega, r(n,\omega)y).$$
This is a skew-product transfomation. For $n \in \Z$ define
\begin{equation*}
  \begin{array}{rcrcl}
    \phi_n&:&G^\Z \times Y &\longrightarrow &G^\N \times Y\\
    & & (\omega,y)&\longmapsto     &((\omega_n,\omega_{n+1},\ldots), y),
\end{array}
\end{equation*}
and let $\cP_\Z(G^\Z\times Y|\cmu^\Z)$ be the set of all
$T$-invariant Borel probability measures $\lambda$ on $G^\Z \times
Y$ such that
\begin{align}
  \label{eq:lambda}
  \phi_{1*}(\lambda) = \cmu^\N \times \nu
\end{align}
for some $\nu \in \cP(Y)$. Observe that because $\lambda$ is $T$-invariant, (\ref{eq:lambda}) implies the projection of $\lambda$ to the first coordinate is $\cmu^\Z$. So $\cP_\Z(G^\Z\times Y|\cmu^\Z) \subset \cP(G^\Z\times Y|\cmu^\Z)$. We give it the subspace topology.

 The main result of this section is:
\begin{thm}\label{thm:affine-homeo}
  $\cP_{\mu}(Y)$ is affinely homeomorphic with $\cP_\Z(G^\Z\times
  Y|\cmu^\Z)$. More precisely, define
  \begin{eqnarray*}
    \alpha: \cP_{\mu}(Y) \to \cP_\Z(G^\Z\times Y|\cmu^\Z)\\
    \alpha(\nu) = \int \delta_\omega \times \nu_\omega~d\cmu^\Z(\omega)
  \end{eqnarray*}
  where
  $$\nu_\omega = \lim_{n\to\infty} r(-n,\omega)^{-1}_*\nu =
  \lim_{n\to\infty}\omega_0^{-1}\omega_{-1}^{-1} \cdots
  \omega_{-(n-1)}^{-1}\nu$$
  for the full measure subset of $G^\Z$ for
  which this limit exists. Then $\alpha$ is an affine
  homeomorphism. 
\end{thm}
  
We need the following lemma.

\begin{lem}
  \label{lem:shift-space}
  \begin{enumerate}

  \item
    \label{item:one-to-one} 
    For all $\lambda_1,\lambda_2 \in \cP_\Z(G^\Z\times Y|\cmu^\Z)$
    \begin{align*}
      \phi_{1*}(\lambda_1) = \phi_{1*}(\lambda_2) \quad \Rightarrow
      \quad \lambda_1=\lambda_2.
    \end{align*}

  \item Given $\lambda \in \cP_\Z(G^\Z\times Y|\cmu^\Z)$ there exists
    a measurable map $\omega \mapsto \lambda^\omega$ from $G^\Z$ into
    $\cP(Y)$ such that
    \begin{align}
      \label{eq:stationary}
      \lambda = \int \delta_\omega \times
      \lambda^\omega~d\cmu^\Z(\omega).
    \end{align}
    Moreover, this map is unique up to null sets, and $\lambda^\omega$
    depends only on $\{w_n\,:\,n \leq 0\}$ for a.e.\ $\omega$.

  \item
    \label{item:stationary} 
    If $\lambda \in \cP_\Z(G^\Z \times Y|\cmu^\Z)$ and
    $\phi_{1*}(\lambda) = \cmu^\N \times \nu$ then $\nu \in
    \cP_{\mu}(Y)$.
  \end{enumerate}

\end{lem}
\begin{proof}
  \begin{enumerate}

  \item Let $\cA_n$ be the sub-sigma-algebra generated by $\phi_n$, so
    that $\sigma(\cup_n\cA_n)$ is the entire sigma-algebra. For $A \in
    \cA_n$ and $i=1,2$, $\lambda_i(A) = \lambda_i(T^{-n+1}A)$, by
    $T$-invariance. But $T^{-n+1}A$ is $\cA_1$-measurable, and, for
    sets in $\cA_1$, $\lambda_1$ and $\lambda_2$ are identical,
    by~\eqref{eq:lambda}.

  \item Existence and uniqueness follow from the disintegration
    theorem.  By~\eqref{eq:lambda}, $\lambda^\omega$ depends only on
    $\{w_n\,:\,n \leq 0\}$.

  \item By~\eqref{eq:stationary}
    \begin{align*}
      \phi_{1*}(T_*\lambda) &= \phi_{1*}\left(\int
        \delta_{\sigma\omega} \times
        \omega_{1*}^{-1}\lambda^\omega~d\cmu^\Z(\omega)\right)\\
      &= \int \delta_{(\omega_2,\omega_3,\ldots)} \times
      \omega_{1*}^{-1}\lambda^\omega~d\cmu^\Z(\omega).
    \end{align*}
    Since $\lambda^\omega$ depends only on $\{\omega_n\,:\,n \leq 0\}$
    then
    \begin{align*}
      &=\int\delta_{(\omega_2,\omega_3,\ldots)}~d\cmu^\Z(\omega) \times
      \int \omega'^{-1}_*\left(\int \lambda^\omega~d\cmu^\Z(\omega)\right)~d\cmu^\Z(\omega')\\
      &= \cmu^\N \times \int
      g_*\nu~d\mu(g),
    \end{align*}
    where the last equality follows from the fact that $\nu = \int
    \lambda^\omega~\d\cmu^\Z(\omega)$, a consequence
    of~\eqref{eq:lambda} and the definition of $\lambda^\omega$.  But
    $T_{*}\lambda = \lambda$, and so
    \begin{align*}
      \nu = \int g_*\nu~d\mu(g).
    \end{align*}
  
  \end{enumerate}
  
\end{proof}

\begin{proof}[Proof of Theorem~\ref{thm:affine-homeo}]
  If $\omega \in G^\Z$ is chosen at random with law $\cmu^\Z$ then
  $n\mapsto r(-n,\omega)^{-1}_*\nu$ is a martingale. By the Martingale
  Convergence Theorem, $\nu_\omega$ exists for a.e.\ $\omega$, and
  furthermore
  \begin{align}
    \label{eq:nu}
    \int\nu_\omega~d\cmu^\Z(\omega) = \nu.
  \end{align}
  Also,
  \begin{align}
    \label{eq:r-nu}
    r(m,\omega)_*\nu_\omega = r(m,\omega)_*\lim_{n\to\infty}
    r(-n,\omega)^{-1}_*\nu
    = \lim_{n\to\infty}r(-n,\sigma^m\omega)^{-1}_*\nu
    =\nu_{\sigma^m\omega},
  \end{align}
  where the second equality follows from the weak* continuity of the
  $G$-action on $\cP(Y)$ and the cocycle property of $r$.

 Recall 
$$ \alpha(\nu) = \int \delta_\omega \times \nu_\omega~d\cmu^\Z(\omega).$$
So $\alpha(\nu)$ is indeed $T$-invariant
  by~\eqref{eq:r-nu}, and, since $\nu_\omega$ depends only on
  $\{\omega_n\,:\,n \leq 0\}$, \eqref{eq:lambda} is satisfied and so
  $\alpha(\nu) \in \cP_\Z(G^\Z\times Y|\cmu^\Z)$.  Also define
  $\beta:\cP_\Z(G^\Z\times Y|\cmu^\Z) \to \cP_{\mu}(Y)$ by
  $$\beta(\lambda) = \int \lambda^\omega ~d\cmu^\Z(\omega),$$
  where $\lambda^\omega$ is given by~\eqref{eq:stationary}. The image of
  $\beta$ is indeed in $\cP_{\mu}(Y)$ by
  Lemma~\ref{lem:shift-space}~(\ref{item:stationary}). 

  Note that $\beta(\lambda)$ is simply the push-forward of $\lambda$
  under the projection on the second coordinate. Hence it follows
  from~\eqref{eq:nu} that $\beta \circ \alpha$ is the identity, and
  $\alpha$ is one-to-one. By
  Lemma~\ref{lem:shift-space}~(\ref{item:one-to-one}) $\beta$ is
  one-to-one, and so $\alpha \circ \beta$ is also the identity. Thus
  $\alpha$ and $\beta$ are inverses.

  It is clear that $\alpha$ and $\beta$ are affine. So it suffices to
  show they are continuous. In fact, it suffices to show that $\alpha$
  is continuous, since $\cP_{\mu}(Y)$ is compact, and since every
  continuous bijection between compact spaces is a homeomorphism.

  For each $n\in \Z$, let $\pi_n:G^\Z \to G$ be the $n$-th coordinate
  projection. Let $A \subset G^\Z$ be a Borel set contained in the
  sigma-algebra generated by $\{\pi_n:~n\in [-m,m] \cap \Z\}$ where
  $m>0$ is some integer. Because $G$ acts continuously, if $\nu_n \to
  \nu_\infty$ in $\cP_{\mu}(Y)$ then
  $$\alpha(\nu_n)^A  = \int_{A} (\nu_n)_\omega~d\cmu^\Z(\omega) = \int_A (\omega_0^{-1}\omega_{-1}^{-1} \cdots
  \omega_{-m}^{-1})_*\nu_n~d\cmu^\Z(\omega)$$ converges to
  $\alpha(\nu_\infty)^A$ as $n\to\infty$. Because the coordinate
  projections generate the sigma-algebra of $G^\Z$, this shows that
  $\alpha(\nu_n)^A$ converges to $\alpha(\nu_\infty)^A$ for all $A$ in
  a dense subset of the measure algebra of $\cmu^\Z$. So
  $\alpha(\nu_n)$ converges to $\alpha(\nu_\infty)$ by
  Corollary~\ref{cor:weak}. Because $\{\nu_n\}$ is arbitrary, $\alpha$
  is continuous.

\end{proof}

\subsection{Proximal and Poisson measures}
In this section we finish the proof of Theorem \ref{thm:G-deltas}. In order to prove that proximal measures form a $G_\delta$ subset, we need the next lemma. 
 

\begin{lem}
  \label{lem:proximal}
  Let $\cQ$ be the set of all measures $\lambda \in \cP_\Z(G^\Z\times
  Y|\cmu^\Z)$ such that there exists some measurable map $f:G^\Z \to Y$
  such that
  $$\lambda = \int \delta_\omega \times \delta_{f(\omega)}~d\cmu^\Z(\omega).$$
  Then $\nu \in \cP_{\mu}(Y)$ is proximal if and only if $\alpha(\nu)
  \in \cQ$ (where the affine homeomorphism $\alpha:\cP_{\mu}(Y) \to
  \cP_\Z(G^\Z\times Y|\cmu^\Z)$ is as in Theorem \ref{thm:affine-homeo}).
\end{lem}

\begin{proof}
Recall from Theorem~\ref{thm:affine-homeo} that for any $\nu \in \cP_{\mu}(Y)$, 
$$\nu_\omega = \lim_{n\to\infty} r(-n,\omega)^{-1}_*\nu = \lim_{n\to\infty}\omega_0^{-1}\omega_{-1}^{-1} \cdots \omega_{-(n-1)}^{-1}\nu$$ exists for $\cmu^\Z$-a.e.\ $\omega \in G^\Z$. It is well known \cite{furstenberg2009stationary} that $\nu$ is proximal if and only if $\nu_\omega$ is a Dirac measure for a.e.\ $\omega$. So the lemma follows from Theorem~\ref{thm:affine-homeo}.

\end{proof}

\begin{cor}\label{cor:proximal}
$\cP^{proximal}_\mu(Y)$ is a $G_\delta$ subset of $\cP_\mu(Y)$.
\end{cor}

\begin{proof}
By the previous lemma it suffices to show that $\cQ$ is a 
$G_\delta$ subset of $\cP_\Z(G^\Z\times Y|\cmu^\Z)$. 

Let $\cP^{ex}(G^\Z \times Y|\cmu^\Z)$ be the set of extreme points of
  $\cP(G^\Z\times Y|\cmu^\Z)$. Observe that for every $\lambda \in
  \cP^{ex}(G^\Z \times Y|\cmu^\Z)$ there exists some measurable map
  $f:G^\Z \to Y$ such that
  $$\lambda = \int \delta_\omega \times \delta_{f(\omega)}~d\cmu^\Z(\omega).$$
  Therefore
  $$\cQ = \cP^{ex}(G^\Z \times Y|\cmu^\Z) \cap \cP_\Z(G^\Z\times Y|\cmu^\Z).$$
  By Lemma~\ref{thm:ex-G-delta}, $\cP^{ex}(G^\Z \times Y|\cmu^\Z)$ is
  a $G_\delta$ subset of $\cP(G^\Z \times Y|\cmu^\Z)$. Of course,
  $\cP_\Z(G^\Z \times Y|\cmu^\Z)$ is closed in $\cP(G^\Z \times
  Y|\cmu^\Z)$ (because it is compact since it is homeomorphic with $\cP_\mu(Y)$). Since closed sets are $G_\delta$ subsets and
  intersections of $G_\delta$'s are also $G_\delta$'s, this proves
  that $\cQ$ is a $G_\delta$.
\end{proof}

\begin{proof}[Proof of Theorem \ref{thm:G-deltas}]
  By Corollaries~\ref{cor:ergodic} and~\ref{cor:proximal}, it suffices
  to show that
  $$\cP^{Poisson}_\mu(Y) = \cP^{max}_\mu(Y) \cap \cP^{proximal}_\mu(Y)$$
  whenever $H(\mu)<\infty$. This is proven
  in~\cite{kaimanovich1983random}.
\end{proof}

\section{Dense measure conjugacy classes}
\label{sec:density}
Let $G \cc Y$ be an action by homeomorphisms on a compact metrizable
space $Y$. Also let $\bfz=G \cc (Z,\zeta)$ be a $\mu$-stationary
action. Let $\Fact(\bfz,Y)$ be the set of all probability measures
$\nu \in \cP_\mu(Y)$ such that $\nu = \pi_*\zeta$ where $\pi:Z \to Y$ is a
$G$-equivariant Borel map. This is the set of all {\em factor
  measures} of the action $G \cc (Z,\zeta)$. Denote by $\Conj(\bfz,
Y)$ the set of all probability measures $\nu \in \cP_\mu(Y)$ such that
$G \cc (\nu, Y)$ is measurably conjugate to $(Z,\zeta)$. This is a
subset of $\Fact(\bfz, Y)$.

The main technical result of this section is
\begin{thm}
  \label{thm:0-1-dense}
  Let $W$ be a compact metric space.  Let $\bfb=G \cc (B,\nu)$ be any
  stationary, essentially free extension of the Poisson boundary. Then
  $\Fact(\bfb,W^G)$ is dense in $\cP_\mu(W^G)$.
\end{thm}

Note that such a $(B,\nu)$ always exists; for example, take the
product of the Poisson boundary with a Bernoulli shift.  Before
proving this claim we draw a number of consequences, and prove a
``sharpness'' claim. Let $\X = \{0,1\}^\N$ be the Cantor space,
equipped with the usual product topology.
\begin{thm}
  \label{thm:poisson-dense}
  Let $\bfb=G \cc (B,\nu)$ be a stationary, essentially free extension
  of the Poisson boundary. Then $\Conj(\bfb, \X^G)$ is dense in
  $\cP_\mu(\X^G)$.
\end{thm}
\begin{proof}
  By Theorem~\ref{thm:0-1-dense} it suffices to show that for every
  factor $\pi:(B,\nu) \to (\X^G,\pi_*\nu)$ there exists a sequence of
  measure-conjugacies $\Phi_i:(B,\nu) \to (\X^G,\Phi_{i*}\nu)$ such
  that
  $$\lim_{i\to\infty} \Phi_{i*}\nu = \pi_*\nu$$
  in the weak* topology on $\cP_\mu(\X^G)$.  For this purpose, choose
  a measure-conjugacy $\psi:(B,\nu) \to (\X^G,\psi_*\nu)$; the
  existence of such a measure-conjugacy follows from the fact that
  there exists a countable dense subset of the measure algebra of
  $(B,\nu)$. We are requiring that $\psi$ is $G$-equivariant and a
  measure-space isomorphism. Define $\Phi_i \colon B \to \X^G$ by
  $$\Phi_i(b)(g) = (x_1,x_2,\ldots)$$
  where $x_j = \pi(b)(g)_j$ if $j\le i$ and $x_j = \psi(b)(g)_{j-i}$
  if $j>i$. If $\Proj_i\colon \X^G \to \X^G$ denotes the projection
  $$\Proj_i(x)(g)=(x_{i+1}(g),x_{i+2}(g),\ldots)$$
  then $\Proj_i \circ \Phi_i = \psi$. Hence $\Phi_i$ is an
  isomorphism. It is clear that $\lim_{i\to\infty} \Phi_{i*}\nu =
  \pi_*\nu.$
\end{proof}
An immediate consequence is the following.
\begin{cor}
  \label{cor:poisson-free-dense}
  If the action of $G$ on its Poisson boundary is essentially free,
  then $\cP^{Poisson}_\mu(\X^G)$, the set of all measures $\lambda \in
  \cP_\mu(\X^G)$ such that $G \cc (\X^G,\lambda)$ is
  measurably-conjugate to the Poisson boundary, is dense in
  $\cP_\mu(\X^G)$.
\end{cor}
See \S~\ref{sec:freeness} for a discussion of conditions which
guarantee freeness of the action of $G$ on its Poisson boundary. Next,
we observe that Theorem~\ref{thm:0-1-dense} is in a sense ``best
possible'':
\begin{prop}
  Let $\bfb=G \cc (B,\nu)$ be a stationary action. Suppose that either
  this action is not essentially free or does not have maximum
  $\mu$-entropy. Then $\Fact(\bfb,\X^G)$ is not dense in
  $\cP_\mu(\X^G)$.
\end{prop}
In particular, when the action on the Poisson boundary is not free,
then the measure conjugates of the Poisson boundary are not dense.
\begin{proof}
  If $\bfb$ does not have maximum $\mu$-entropy then because
  $\mu$-entropy is lower-semicontinuous (see the proof of
  Corollary~\ref{cor:ergodic}), $\overline{\Fact(\bfb,\X^G)} \cap
  \cP^{max}_\mu(\X^G) = \emptyset$. However, $\cP^{max}_\mu(\X^G)$ is
  nonempty, since it includes a Poisson measure (see,
  e.g.,~\cite{tserunyan2012finite}).

  Suppose $\bfb$ is not essentially free. Let $\cP^{free}_\mu(\X^G)$
  be the set of all $\eta \in \cP_\mu(\X^G)$ such that $G \cc
  (\X^G,\eta)$ is essentially free. We claim that
  $\overline{\Fact(\bfb,\X^G)} \cap \cP^{free}_\mu(\X^G) =
  \emptyset$.

  Let $g\in G \setminus \{e\}$ be an element such that
  $\nu(\Fix(g:B))>0$ where $\Fix(g:B)=\{b\in B:~gb=b\}$. Observe that
  if $\lambda \in \Fact(\bfb,\X^G)$ then $\nu(\Fix(g:\X^G)) \ge
  \nu(\Fix(g:B))>0$. So it suffices to show that the map $\lambda
  \mapsto \lambda(\Fix(g:\X^G))$ is an upper semi-continuous function
  of $\lambda \in \cP_\mu( \X^G)$. This follows from the portmanteau
  Theorem because $\Fix(g:\X^G)$ is a closed subset of $\X^G$, and
  because the action $G \cc \X^G$ is continuous.
\end{proof}

The remainder of this section is devoted to the proof of
Theorem~\ref{thm:0-1-dense}. We begin with some preliminaries.

\subsection{Outline of the proof of Theorem~\ref{thm:0-1-dense}}
\label{sec:outline}

The proof of Theorem~\ref{thm:0-1-dense} uses a technique analogous to
painting names on Rohlin towers. First we use the amenability and
freeness of the action $G \cc (B,\nu)$ to show that there exist
partitions $\{\cQ_{n,b}\}_{n\in \N, b \in B}$ of $G$ satisfying:
\begin{itemize}
\item for each $n$, the assignment $b \mapsto \cQ_{n,b}$ is measurable
  and $G$-equivariant,
\item each partition element of $\cQ_{n,b}$ is finite, 
\item if $\cQ_{n,b}(g) \subset G$ denotes the partition element of $\cQ_{n,b}$ containing $g\in G$ then for every finite subset $F \subset G$, 
$$\lim_{n\to\infty} \nu(\{b\in B:~ F \subset \cQ_{n,b}(1_G)\})=1,$$
where $1_G$ is the identity element of $G$.
\end{itemize}
This sequence carries information analogous to a Rohlin tower. For
technical reasons, it is also useful to show the existence of subsets
$R_{n,b} \subset G$ such that
\begin{itemize}
\item the assignment $b \mapsto R_{n,b}$ is measurable and
  $G$-equivariant,
\item $R_{n,b}$ contains exactly one element from each partition element of $\cQ_{n,b}$.
\end{itemize}
Elements of $R_{n,b}$ are {\em roots} of the partition elements of $\cQ_{n,b}$.

The partitions $\cQ_{n,b}$ and subsets $R_{n,b}$ induce a natural partition of a special subset of $B$. Namely, we define
\begin{itemize}
\item $B_n = \{b \in B:~ 1_G \in R_{n,b} \}$,
\item $\cP_n$ to be the partition of $B_n$ defined by: $b,b' \in B_n$ are in the same partition element of $\cP_n$ if and only if $\cQ_{n,b}(1_G) = \cQ_{n,b'}(1_G)$,
\item $\psi_n:\cP_n \to 2^G$ by $\psi_n(P) = \cQ_{n,b}(1_G)$ for any $b \in P$.
\end{itemize}

Now let $\theta \in \cP_\mu(W^G)$. It suffices to construct $G$-equivariant measurable maps $\pi_n:B \to W^G$ such that 
$$\lim_{n\to\infty} \pi_{n*} \nu = \theta.$$
To achieve this, we first decompose $\theta$ and $\nu$ using the
Poisson boundary. To be precise, let $\alpha:B \to \Pi(G,\mu)$ be a
factor map to the Poisson boundary and define a map $B \to
\cP(W^G)$, $b \mapsto \theta_b$ by
\begin{eqnarray}\label{eqn:theta_b}
\theta_b = \lim_{n\to\infty} g_n\theta
\end{eqnarray}
where $\{g_n\}$ is any sequence in $G$ with
$\bnd(\{g_n\})=\alpha(b)$. Here $\bnd$ is the map that assigns to
almost every sequence in $(G^\N,\cP_\mu)$ the corresponding point in
the Poisson boundary. This is well defined since $\lim_{n\to\infty}
g_n\theta$ is measurable in the shift-invariant sigma-algebra of
$(G^\N,\cP_\mu)$, and therefore depends only on $\bnd(\{g_n\})$.

Because the map $\pi_n$ must be $G$-equivariant, it suffices to define
it on the subset $B_n$ (since this subset intersects every $G$-orbit
nontrivially). To define $\pi_n(b)$ (for $b\in B_n$), the rough idea
is to take an element $x \in W^G$ which is ``typical'' with respect to
$\theta_b$ and choose $\pi_n(b)(g)=x(g)$ for $g \in \psi_n(b)$. We use
equivariance to define $\pi_n(b)$ on the rest of $G$. The element $x$
should be a measurable function of the element $b$. This means we must
choose a map $\beta_{n,P}:P \to W^G$ (for each $P \in \cP_n$) such
that
$$\beta_{n,P*}(\nu \res P) \sim \int_P \theta_b~d\nu(b)$$
where $\sim$ means close in total variation norm. Actually, this is not good enough because we need a good approximation on translates of $B_n$. So what we really require is that
\begin{eqnarray}\label{eqn:key-dense}
  g^{-1}_*\beta_{n,P*}(g_*\nu \res P) \sim \int_{g^{-1}P} \theta_b~d\nu(b)
\end{eqnarray}
for all $g\in \psi_n(P)$. Then we define $\pi_n(b)(g)=\beta_{n,P}(g)$ for $b\in P$ and $g \in \psi_n(b)$. It remains only to verify that $\pi_n$ has the required properties.


\subsection{A random rooted partition from amenability}

Let $\Part(G)$ be the set of all (unordered) partitions of $G$. We may
identify a partition $\cP \in \Part(G)$ with the equivalence relation
that it determines. Any equivalence relation on $G$ is a subset of $G\times G$ and therefore may be identified with an element of
$2^{G\times G}$ which is a compact metric space in the product
topology. So we may view $\Part(G)$ as a subset of $2^{G\times G}$ and
give it the subspace topology. In this topology, it is a compact
metrizable space. Also, $G$ acts continuously on $\Part(G)$ by $g\cP =
\{gP:~P \in \cP\}$.

\begin{lem}
  \label{lem:rohlin}
  Let $G \cc (B,\nu)$ be an essentially free, measure preserving
  extension of the Poisson boundary. Then there exist partitions
  $\cQ_{n,b} \in \Part(G)$ and subsets $R_{n,b} \subset G$ that have
  the properties detailed in \S~\ref{sec:outline}.
\end{lem}

In fact, as the proof below shows, this lemma holds for any
essentially free amenable $G$-action.

\begin{proof}

The action of $G$ on its Poisson boundary is amenable in Zimmer's sense~\cite{zimmer1978amenable}. Because extensions of amenable actions are amenable,  $G\cc (B,\nu)$ is amenable.
 
 Let $E_G \subset B \times B$ denote the equivalence relation
  $$E_G=\{ (b,gb):~b\in B, g\in G\}.$$
 Because the action of $G$ on $(B,\nu)$ is amenable, $E_G$ is
  hyperfinite. This means that there exists a sequence
  $\{\cR_n\}_{n=1}^\infty$ of Borel equivalence relations $\cR_n
  \subset B \times B$ such that
  \begin{itemize}
  \item for a.e. $b \in B$ and every $n$, the $\cR_n$-class of $b$,
    denoted $[b]_{\cR_n}$, is finite,
  \item $\cR_n \subset \cR_{n+1}$ for all $n$,
  \item $E_G = \cup_n \cR_n$.
  \end{itemize}
  
  Define $\cQ_{n,b}$ by: $g_1,g_2$ are in the same part of $\cQ_{n,b}$ if and only if $g_1^{-1}b \cR_n g_2^{-1}b$. Because $\cR_n$ is Borel, the assignment $b \mapsto \cQ_{n,b}$ is also Borel, hence measurable. For any $h\in G$, $g_1,g_2$ are in the same part of $\cQ_{n,hb}$ if and only if $g_1^{-1}hb \cR_n g_2^{-1}hb$ which occurs if and only if $h^{-1}g_1, h^{-1}g_2$ are in the same part of $\cQ_{n,b}$.  So $\cQ_{n,hb} = h\cQ_{n,b}$ as required.
  
 Because each $\cR_n$-class is finite, each part of $\cQ_{n,b}$ is finite. Let $F \subset G$ be finite. Because $E_G = \cup_n \cR_n$ is an increasing union the probability that, for a randomly chosen $b\in B$, $\{fb\}_{f\in F}$ is contained in an $\cR_n$-equivalence class, tends to 1 as $n\to\infty$. Equivalently,
 $$\lim_{n\to\infty} \nu(\{b\in B:~ F \subset \cQ_{n,b}(1_G)\})=1,$$
 and so the $\cQ_{n,b}$ have the required properties.
 
 As another consequence of the fact that each $\cR_n$-class is finite,
 there exist measurable subsets $B_n \subset B$ such that for a.e.\
 $b\in B$, $B_n \cap [b]_{\cR_n}$ contains exactly one element. So
 define $R_{n,b} := \{g\in G:~g^{-1}b \in B_n\}$. To check that
 $R_{n,b}$ contains exactly one element from each partition element of
 $\cQ_{n,b}$, observe that if $P$ is any part of $\cQ_{n,b}$ and $g\in
 P$ then there exists a unique $b' \in B_n$ such that $g^{-1}b \cR_n
 b'$. Because the action of $G$ on $B$ is essentially free, there is a
 unique $g_0 \in G$ such that $b' = g_0^{-1}b$. So $g_0$ is the unique
 element of $R_{n,b} \cap P$.
 \end{proof}

\subsection{Painting names}
 
Following the outline in \S \ref{sec:outline}, we let $\theta \in
\cP_\mu(W^G)$ and define $\theta_b$ (for $b\in B$) by equation
(\ref{eqn:theta_b}). In this subsection, we choose $\beta_{n,b}$ to
satisfy (\ref{eqn:key-dense}). First we need to recall some basic
facts about total variation distance.

Let $\lambda_1,\lambda_2$ be Borel measures on a space $Z$. Their
{\em total variation distance} is defined by
$$\| \lambda_1 - \lambda_2 \| = \sup_A |\lambda_1(A) - \lambda_2(A)|$$
where the supremum is over all Borel subsets $A \subset Z$. We will
need two elementary facts:
\begin{claim}
  \label{defn:tv}
  \begin{enumerate}
  \item Suppose there exists a measure $\lambda_3$ such that
    $\lambda_1,\lambda_2$ are both absolutely continuous to
    $\lambda_3$. Then
    $$\| \lambda_1 - \lambda_2\| \le \int \left| \frac{d\lambda_1}{d\lambda_3}(z)- \frac{d\lambda_2}{d\lambda_3}(z)\right| ~d\lambda_3(z).$$
  \item If $\Phi:Z \to Z'$ is any Borel map then $\|\Phi_*\lambda_1 - \Phi_*\lambda_2\| \le \|\lambda_1 - \lambda_2\|$.
  \end{enumerate}
\end{claim}

Define $B_n,\cP_n,\psi_n$ as in \S \ref{sec:outline}.

\begin{lem}\label{lem:qP}
  For every $P \in \cP_n$ and $\delta_P>0$ there exists a measurable
  map $\beta_{n,P}:P \to W^G$ such that
  $$\left\|g^{-1}_*\beta_{n,P*}(g_*\nu \res P) - \int_{g^{-1}P} \theta_b~d\nu(b)\right\| < \delta_P$$
  for every $g\in \psi_n(P)$.
\end{lem}

\begin{proof}
  Let $\cK$ be a countable partition of $P$ such that for every $g\in
  \psi_n(P)$
  $$\sum_{K\in \cK} \int_K \left| \frac{dg_*\nu}{d\nu}(b) - C(g,K) \right|~d\nu(b) < \delta_P/2$$
  for some constants $\{C(g,K):~g\in \psi_n(P), K \in \cK\}$. Such a
  partition exists because step functions are dense in
  $L^1(B,\nu)$. By fact (1) of Claim~\ref{defn:tv},
  $$\sum_{K \in \cK} \|  (g_*\nu \res K) - C(g,K)(\nu \res K)  \| < \delta_P/2.$$
  Choose a measurable map $\beta_{n,P}:P \to W^G$ so that
  $$\beta_{n,P*}(\nu \res K) = \int_K \theta_b~d\nu(b)$$
  for every $K \in \cK$. Here we are using the fact that $(B,\nu)$ has
  no atoms, which holds because the action $G \cc (B,\nu)$ is
  stationary and essentially free, and because $G$ is countably
  infinite\footnote{A finite, essentially free stationary measure
    cannot be atomic, since in any finite stationary measure the
    finite set of atoms of maximal measure must be invariant, and thus
    have non-trivial stabilizers.}
  . 

  Then for any $g\in \psi_n(P)$,
  \begin{eqnarray*}
    &&\left\|\beta_{n,P*}(g_*\nu \res P) - \int_P \theta_b~dg_*\nu(b) \right\| \\
    &\le& \sum_{K\in \cK} \left\|\beta_{n,P*}(g_*\nu \res K) - \int_K \theta_b~dg_*\nu(b)\right\| \\
    &\le& \sum_{K\in \cK} \left\|\beta_{n,P*}(g_*\nu \res K) - C(g,K) \beta_{n,P*}(\nu \res K)\right\| + \left\|C(g,K)\beta_{n,P*}(\nu \res K) - C(g,K)\int_K \theta_b~d\nu(b)\right\| \\
    &&\quad + \left\|C(g,K)\int_K \theta_b~d\nu(b) - \int_K \theta_b~dg_*\nu(b)\right\| \\
    &<& \delta_P/2 + 0 + \delta_P/2 = \delta_P.
  \end{eqnarray*}
  The lemma now follows from fact (2) of Claim~\ref{defn:tv} and
  $$g^{-1}_*\left(\int_P \theta_b~dg_*\nu(b)\right) = \int_{g^{-1}P}\theta_b ~d\nu(b).$$

\end{proof}

\subsection{End of the proof}

Define $\pi_n(b)(g)=\beta_{n,P}(g)$ for $b\in P \in \cP_n$ and $g \in
\psi_n(b)$. 

\begin{lem}
There exists a unique $G$-equivariant extension of $\pi_n$ from $B$ to $W^G$.
\end{lem}

\begin{proof}


Suppose $h \in G$, $b\in B_n$. We will define $\pi_n(b)(h)$ as follows. Let $Q \in \cQ_{n,b}$ be the part containing $h$. There exists a unique $g\in R(n,b) \cap Q$. Because $R_{n,g^{-1}b} = g^{-1}R_{n,b} \ni 1_G$, it follows that $g^{-1}b \in B_n$. By definition, $\psi_n(g^{-1}b)$ is the part of $\cQ_{n,g^{-1}b}=g^{-1}\cQ_{n,b}$ containing $1_G$. So $\psi_n(g^{-1}b) =g^{-1}Q$ contains $g^{-1}h$. So $\pi_n(g^{-1}b)(g^{-1}h)$ is well-defined. We now define
$$\pi_n(b)(h) = \pi_n(g^{-1}b)(g^{-1}h)$$
and observe that we have now defined $\pi_n$ on $B_n$ so that $\pi_n(gb)=g\pi_n(b)$ whenever $g\in G$ and $b,gb\in B_n$. 

Next we define $\pi_n$ for arbitrary $b\in B$ as follows. Let $Q$ be the part of $\cQ_{n,b}$ with $1_G \in Q$. Let $g\in G$ be the unique element in $R_{n,b} \cap Q$. Because $R_{n,g^{-1}b} = g^{-1}R_{n,b} \ni 1_G$, it follows that $g^{-1}b \in B_n$. So $\pi_n(g^{-1}b)$ is well-defined. We now define $\pi_n(b) = g\pi_n(g^{-1}n)$. This is the unique $G$-equivariant extension of $\pi_n$.
\end{proof}

\begin{lem}\label{lem:partition}
  $\{g^{-1}P\,:\,P \in \cP_n, g \in \psi_n(P)\}$ is a partition of $B$  (up to measure zero).
\end{lem}
The proof of this lemma follows directly from the definitions, and
hence we omit it.

\begin{lem}
  \label{clm:total-variation}
  If $P \in \cP_n$, $F=\psi_n(P) \subset G$ and $g \in F$ then
  $$\left\|\Proj_{W^{g^{-1}F}}(\pi_{n*}(\nu \res g^{-1}P)) - \Proj_{W^{g^{-1}F}}(\theta_{g^{-1}P})\right\|< \delta_P$$
  where $\theta_P = \int_P \theta_b~d\nu(b).$
\end{lem}

\begin{proof}
Let $b \in P$ be random with law $\nu \res g^{-1}P$
  (normalized to have mass 1). Because $\pi_n$ is $G$-equivariant,
  $$\pi_n(b)(g^{-1}f)=g\pi_n(b)(f)=\pi_n(gb)(f)\quad \forall f\in F=\psi_n(P).$$ 
  Because $gb \in P$,   
  $$\pi_n(gb)(f)=\beta_{P,n}(gb)(f)=g^{-1}\beta_{P,n}(gb)(g^{-1}f).$$
  Thus $\pi_n(b)(h)= g^{-1}\beta_{P,n}(gb)(h)$ for all $h\in g^{-1}F$.  The claim now follows from Lemma
  \ref{lem:qP}.
\end{proof}

 Let $\epsilon>0$ and $F\subset G$ be finite such that $1_G\in F$. It suffices to show that
$$\limsup_{n\to\infty}\| \Proj_{W^F}(\pi_{n*}\nu) - \Proj_{W^F}\theta\| \le 3\epsilon.$$
By Lemma \ref{lem:rohlin} there exists an $N$ such that $n>N$ implies
$$ \nu(\{b\in B:~ F \subset \cQ_{n,b}(1_G)\}) >1-\epsilon.$$
For $P\in \cP_n$, let $\psi'_n(P)=\{g\in \psi_n(P):~g^{-1}\psi_n(P)
\supset F\}$. Let
$$\GOOD(n)= \cup \{ g^{-1}P:~P\in \cP_n, g\in \psi'_n(P)\}.$$
\begin{lem}\label{lem:GOOD}
  $\nu(\GOOD(n)) >1-\epsilon$.   
\end{lem}
\begin{proof}
By Lemma \ref{lem:rohlin} there exists $N$ such that $n>N$ implies
$$\nu(\{b\in B:~ F \subset \cQ_{n,b}(1_G)\})>1-\epsilon.$$
Suppose $b \in B$ is such that $F \subset \cQ_{n,b}(1_G)$. By Lemma \ref{lem:partition} there exists unique  $P \in \cP_n$ and $g \in \psi_n(P)$ such that $b \in g^{-1}P$. It suffices to show that $g\in \psi'_n(P)$, i.e., $F \subset g^{-1}\psi_n(P)$. 


Because $gb \in P \subset B_n$, $1_G \in R_{n,gb}=gR_{n,b}$. So $g^{-1} \in R_{n,b}$. Let $Q$ be the part of $\cQ_{n,b}$ containing $1_G$. So $F \subset Q$. Now $gQ \ni g \in \psi_n(P)$. So $gQ$ is the part of $g\cQ_{n,b}=\cQ_{n,gb}$ containing $g$. Because $gb \in P$, $\psi_n(P)$ is a part of $\cQ_{n,gb}$. By hypothesis $g\in \psi_n(P)$. So  $gQ=\psi_n(P)$. Since $F \subset Q$, this implies $F \subset g^{-1}\psi_n(P)$ as claimed.

\end{proof}

We now have:
\begin{eqnarray*}
  \|\Proj_{W^F}(\pi_{n*}\nu) - \Proj_{W^F}\theta \| &\le&  \sum_{P\in \cP_n}\sum_{g\in \psi_n(P)} \|\Proj_{W^F}(\pi_{n*}(\nu \res g^{-1}P)) - \Proj_{W^F}(\theta_{g^{-1}P})\|\\
  &\le& 2\epsilon +  \sum_{P\in \cP_n}\sum_{g\in \psi'_n(P)}  \|\Proj_{W^F}(\pi_{n*}(\nu \res g^{-1}P)) - \Proj_{W^F}(\theta_{g^{-1}P})\|\\
  &\le& 2\epsilon +  \sum_{P\in \cP_n}\sum_{g\in \psi'_n(P)} \delta_P. 
\end{eqnarray*}
The first inequality is implied by Lemma \ref{lem:partition}, the second by Lemma \ref{lem:GOOD} and the last by Lemma \ref{clm:total-variation}. We may choose each $\delta_P$ so that $\sum_{P\in \cP_n}\sum_{g\in
  \psi'_n(P)} \delta_P < \epsilon$. Since $\epsilon, \theta, F$ are arbitrary,
this implies Theorem \ref{thm:0-1-dense}.

\section{Freeness of the Poisson boundary action} \label{sec:freeness}


In this section we prove that every torsion-free non-elementary
hyperbolic group acts essentially freely on its Poisson boundary. For standard references on hyperbolic groups see \cite{gromov-1987, coornaert-delzant-papadopolous, notes-on-hyperbolic-groups, ghys-de-la-harpe}. Our main result is a consequence of the following more general result:

\begin{thm}
Let $G$ be a countable group with a generating probability measure $\mu$. Let $G \cc \Pi(G,\mu)$ denote the action of $G$ on its Poisson boundary. Suppose $G$ has only countably many amenable subgroups. Then there exists a normal amenable subgroup $N \vartriangleleft G$ such that the stabilizer of almost every $x \in \Pi(G,\mu)$ is equal to $N$. 
\end{thm}

\begin{proof}
Let $(B,\nu)= \Pi(G,\mu)$ be the Poisson boundary. Let $\Sub_G$ be the
space of subgroups of $G$, equipped with the Fell or Chabauty
topology; in our case of discrete groups this is the same topology as
the subspace topology inherited from the product topology on the space
of subsets of $G$ (see, e.g.,~\cite{abert2012kesten}). Let $\stab
\colon B \to \Sub_G$ be the stabilizer map given by $\stab(b) = \{g
\in G\,:\, gb = b\}$. This map is $G$-equivariant, and so $\stab_*\nu$
is a $\mu$-stationary distribution on $\Sub_G$.

Since the action $G \cc (B,\nu)$ is amenable, $\stab_*\nu$ is
supported on amenable groups (\cite[Theorem 2]{golodets1990amenable},
or, e.g.,~\cite[Theorem A (v)]{adams-elliott-giordano}). However, $G$
contains only countably many amenable subgroups so $\stab_*\nu$ has
countable support. However, every countably supported stationary
measure is invariant; this is because the set of points that have
maximal probability is invariant.

Finally, every invariant factor of the Poisson boundary is trivial
(see, e.g.,~\cite[Corollary 2.20]{bader2006factor}), and so
$\stab_*\nu$ is supported on a single subgroup, which has to be
amenable and normal.
\end{proof}

\begin{proof}[Proof of Proposition~\ref{prop:free-boundary}]
  By the previous theorem, it suffices to prove:
  \begin{enumerate}
  \item $G$ has only countably many amenable subgroups,
  \item $G$ does not contain a nontrivial normal amenable subgroup.
  \end{enumerate}
  It is well-known that hyperbolic groups satisfy a strong form of the
  Tits' Alternative: every subgroup is either virtually cyclic or
  contains a nonabelian free subgroup \cite{gromov-1987}. In
  particular, every amenable subgroup is virtually cyclic and
  therefore must be finitely generated.  Since there are only
  countably many finitely generated subgroups of $G$, this proves (1).

  It is also well-known that any normal amenable subgroup of a
  non-elementary word hyperbolic group must be finite. To see this, recall the Gromov compactification $\bar{G}$ of $G$. Let $\partial G = \bar{G} - G$ and for any subgroup $H\le G$, let $\partial H = \bar{H}-H$  where $\bar{H}$ is the closure of $H$ in $\bar{G}$. It is well-known that if $H$ is virtually infinite cyclic then $\partial H$ consists of two points and $H$ fixes $\partial H$ (for example see \cite[Theorem 12.2 (1)]{kapovich-boundaries}). If $H$ is normal then it follows that $\partial G=\partial H$  \cite[Theorem 12.2 (5)]{kapovich-boundaries}. However, this implies that $G$ is virtually cyclic which contradicts our assumption that $G$ is nonelementary. 
  
  Finally, since we are assuming that $G$ is torsion-free, every
  finite subgroup is trivial.
\end{proof}

On the other hand, it is easy to construct examples of $(G,\mu)$ such
that the action $G \cc \Pi(G,\mu)$ is not essentially free. For
example, this occurs whenever $G$ is nontrivial and abelian since in
this case $\Pi(G,\mu)$ is trivial. For a less trivial example, suppose
$G_1,G_2$ are two countable discrete groups. Let $\mu_i$ be
a generating probability measure on $G_i$. Then $\mu_1\times
\mu_2$ is a generating measure on $G_1\times G_2$. By
\cite[Corollary 3.2]{bader2006factor},
$$G_1\times G_2 \cc \Pi(G_1\times G_2, \mu_1\times\mu_2) \cong G_1\times G_2 \cc \Pi(G_1, \mu_1)\times \Pi(G_2, \mu_2).$$
Therefore if $G_1 \cc \Pi(G_1,\mu_1)$ is not essentially
free then $G_1\times G_2 \cc \Pi(G_1\times G_2,
\mu_1\times\mu_2)$ is also not essentially free. For example, if
$G_1$ is finite or equal to $\Z^d$ for some $d \ge 1$ then this
is always the case.

\section{Nonsingular and stationary transformations and actions}

\subsection{The group of nonsingular transformations} \label{sec:aut}
Let $(X,\nu)$ denote a standard nonatomic Lebesgue probability
space. A measurable transformation $T \colon X \to X$ is {\em
  nonsingular} if $T_*\nu$ is equivalent to $\nu$; equivalently, $\nu$
is $T$-quasi-invariant. Let $\Aut^*(X,\nu)$ be the set of all
nonsingular invertible transformations of $(X,\nu)$ in which we
identify any two transformations that agree almost everywhere. More
precisely: $\Aut^*(X,\nu)$ consists of equivalence classes of
nonsingular transformations in which two such transformations are
considered equivalent if they agree almost everywhere. To each $T \in
\Aut^*(X,\nu)$ and $1 \leq p < \infty$ we assign the isometry $U_{T,p}
\in \Isom(L^p(X,\nu))$ given by
\begin{align*}
  [U_{T,p}f](x) = \left( \frac{d(T_*\nu)}{d\nu}(x)\right)^{1/p}f(T^{-1}x).
\end{align*}
The map $T \mapsto U_{T,p}$ is an algebraic isomorphism of
$\Aut^*(X,\nu)$ with the subgroup of $\Isom(L^p(X,\nu))$ that
preserves the cone of positive functions (see,
e.g.,~\cite[Theorem 3.1]{lamperti1958isometries}).

A topology $\tau$ on $L^p(X,\nu)$ can be lifted to a topology on
$\Isom(L^p(X,\nu))$: a sequence $U_1,U_2,\ldots \in \Isom(L^p(X,\nu))$
converges to $U$ if $U_nf$ converges in $\tau$ to $Uf$ for every $f
\in L^p(X,\nu)$.

In particular, the strong operator topology (SOT) and the weak
operator topology (WOT) on $\Isom(L^p(X,\nu))$ are derived from the
weak and norm (respectively) topologies on $L^p(X,\nu)$ (for $1\le p
<\infty$). In the norm topology, $\lim_n f_n = f$ if
$\lim_n\|f_n-f\|_p=0$. In the weak topology, $\lim_n f_n = f$ if
$\lim_n\langle f_n,g \rangle = \langle f,g \rangle$ for all $g$ in
$L^q(X,\nu)$, where $1/p+1/q=1$.

In~\cite{CK79} it is shown that the topologies induced on
$\Aut^*(X,\nu)$ from the SOT on $\Isom(L^p(X,\nu))$ coincide, for all
$1 \leq p < \infty$. In~\cite[Theorem 2.8]{megrelishvili2001operator}
it is shown that for each $p$ with $1<p<\infty$, the topologies
induced on $\Aut^*(X,\nu)$ from the SOT and the WOT on
$\Isom(L^p(X,\nu))$ coincide, since then $L^p(X,\nu)$ is
reflexive. The topology that these induce on $\Aut^*(X,\nu)$ is called
the {\em weak topology}. In~\cite{CK79} it is shown the weak topology
on $\Aut^*(X,\nu)$ is a Polish group topology, which makes it a
natural choice.

We are also concerned with the topology on $\Aut^*(X,\nu)$ induced
from the WOT on $\Isom(L^1(X,\nu))$. We call this the {\em very weak
  topology}.  This is not a group topology:

\begin{lem}
  The topology on $\Aut^*(X,\nu)$ induced from the WOT on
  $\Isom(L^1(X,\nu))$ is not a group topology. More precisely, the
  multiplication map is not jointly continuous.
\end{lem}

\begin{proof}
  Without loss of generality $X = [0,1]$ and $\nu$ is the Lebesgue
  measure. Below we consider $\Aut^*(X,\nu)$ with the WOT induced from
  its embedding into $\Isom(L^1(X,\nu))$.

  For $n\in \N$, let $T_n \in \Aut^*(X,\nu)$ be the piecewise linear
  map such that
  \begin{itemize}
  \item $T_n$ maps the interval $[k/n, k/n+1/n - 1/n^2]$ linearly to
    the interval $[k/n, k/n + 1/n^2]$ for each integer $k$ with $0\le
    k < n$,

  \item $T_n$ maps $[k/n + 1/n - 1/n^2, k/n + 1/n]$ linearly to the
    interval $[k/2 + 1/n^2, k/n + 1/n]$ for each $k$ with $0 \le k <
    n$.
  \end{itemize}
  Because $T_n$ preserves the intervals $[k/n, k/n + 1/n]$ for each
  $k$, $T_n$ converges to the identity. Let $S_n$ be the following
  measure-preserving map:

  \begin{itemize}
  \item $S_n$ maps the interval $[k/n, k/n + 1/n^2]$ to the interval
    $[k/n + 1/2, k/n + 1/n^2 + 1/2]$ where everything is considered
    mod 1. In other words, $S_n$ behaves like a rotation when
    restricted to these intervals.

  \item $S_n$ fixes all other points. That is: $S_n(x) = x$ if $x$ is
    not contained in any interval of the form $[k/n, k/n + 1/n^2]$.
  \end{itemize}

  The fixed point set of $S_n$ has measure $1 - 1/n$. So $S_n$
  converges to the identity. However, the composition $S_nT_n$ does
  not converge to the identity. Instead it converges to the rotation
  $x \mapsto x + 1/2 \mod 1$. This is because $T_n$ pushes most of the
  interval into the little subintervals $[k/n, k/n + 1/n^2]$ which are
  then rotated by $S_n$.

\end{proof}

We observe that both the weak topology and the very weak topology
coincide on $\Aut(X,\nu)$, the subgroup of measure preserving
transformations. We choose to study the weak topology because it is
the natural Polish group topology on $\Aut^*(X,\nu)$. We also study
the very weak topology since it is strongly related (by Theorem
\ref{thm:generic-equivalence}) to the weak* topology on
$\cP_\mu(\Pol^G)$; the weak topology is too fine for this purpose, and
indeed Theorem~\ref{thm:generic-equivalence} is not true for the weak
topology.

\subsubsection{The weak topology}
A
subbase of open sets for the weak topology on $\Aut^*(X,\nu)$ is
\begin{align*}
  W_{A,\epsilon}(T) = \big\{ S\in \Aut^* (X,\nu ) \csuchthat
  \| U_{T,1}\ind_A- U_{S,1}\ind_A  \|_1 <
  \epsilon \big\}
\end{align*}
where $A\subseteq X$ is measurable, $\epsilon >0$ and  $T\in
\Aut^* (X,\nu )$.

If we let $\cS$ be a countable base for the sigma-algebra of
$(X,\nu)$, then $T_n \to T$, in this topology, if
\begin{align*}
  \lim_n \int_X\left\lvert\frac{dT_{n*}\nu}{d\nu}(x)\ind_{S}(T^{-1}_nx) -
  \frac{dT_{*}\nu}{d\nu}(x)\ind_{S}(T^{-1}x)\right\rvert d\nu(x) = 0.
\end{align*}
for every $S \in \cS$. Equivalently, $T_n \to T$ if $dT_{n*}\nu/d\nu$
converges to $dT_{*}\nu/d\nu$ in $L^1(X,\nu)$, and if furthermore
\begin{align}
  \label{eq:very-weak-conv}
  \lim_n \nu(S_1 \cap T_n S_2) = \nu(S_1 \cap T S_2),
\end{align}
for every $S_1,S_2 \in \cS$.

\subsubsection{The very weak topology}\label{sec:vweak}


A subbase of open sets for the very weak topology on $\Aut^*(X,\nu)$  is
\begin{align*}
  W_{A,B,\epsilon}(T) = \big\{ S\in \Aut^* (X,\nu ) \csuchthat
  |\langle U_{S,1}\ind_A,\ind_B \rangle - \langle U_{T,1}\ind_A,\ind_B \rangle | <
  \epsilon \big\}
\end{align*}
where $A,B\subseteq X$ are measurable, $\epsilon >0$,  $T\in
\Aut^* (X,\nu )$ and $\langle f, g \rangle = \int_Xf(x)
\cdot g(x)d\nu(x)$.

In the very weak topology it is enough that~\eqref{eq:very-weak-conv}
holds for every $S_1,S_2 \in \cS$ to guarantee that $T_n \to T$; the
$L^1$-convergence of the Radon-Nikodym derivatives is not needed.

We note that:
\begin{lem}
  $\Aut^*(X,\nu)$ is a Polish space with respect to the very weak
  topology.
\end{lem}

\begin{proof}
  It suffices to observe that $\Aut^*(X,\nu)$ is closed in
  $\Isom(L^1(X,\nu))$ and $\Isom(L^1(X,\nu))$ is a Polish space with
  respect to the weak operator topology.
\end{proof}

\subsection{Spaces of actions}\label{sec:action-space}
Given a countable group $G$, let $A^*(G,X,\nu)=\Hom(G,\Aut^*(X,\nu))$
denote the set of all homomorphisms $a \colon G \to \Aut^*(X,\nu)$. We
consider $A^*(G,X,\nu)$ as a subset of the product space
$\prod_G\Aut^*(X,\nu)$. Given a topology on $\Aut^*(X,\nu)$, we endow
$\prod_G\Aut^*(X,\nu)$ with the product topology and $A^*(G,X,\nu)$
with the subspace topology.

When $A^*(G,X,\nu)$ has the (very) weak topology, then the resulting
topology on $A^*(G,X,\nu)$ will also be called the (very) weak
topology.

\subsubsection{The space of stationary actions}\label{sec:vweak-space}
When $A^*(G,X,\nu)$ has the very weak topology, then $a_n \to a$ if
\begin{align}
  \label{eq:lim-a}
  \lim_n \nu(S_1 \cap a_n(g) S_2) = \nu(S_1 \cap a(g) S_2),
\end{align}
for every $g \in G$ and measurable $S_1,S_2 \in \cS$. In particular
this implies that
\begin{align}
  \label{eq:lim-a-nu}
  \lim_n a_n(g)_*\nu = a(g)_*\nu
\end{align}
in the weak* topology on $\cP(X)$ (regardless of what compatible
topology $X$ is endowed with).

Given a generating measure $\mu$ on $G$, let $A_\mu(G,X,\nu)$ be the
set of all $a\in  A^*(G,X,\nu)$ such that $a$ is {\em $\mu$-stationary}:
\begin{align*}
  \nu = \sum_{g \in G}\mu(g)  a(g)_*\nu.
\end{align*}
The equation above, together with~\eqref{eq:lim-a-nu}, imply that
$A_\mu(G,X,\nu)$ is closed in $\prod_G \Aut^*(X,\nu)$ when the later is
equipped with the product weak topology. Because the product very weak
topology is weaker than the product weak topology, $A_\mu(G,X,\nu)$ is also
closed as a subspace of $A^*(G,X,\nu)$ with respect to the weak
topology. Hence $A_\mu(G,X,\nu)$ is a Polish space, when equipped with
either the weak or the very weak topology.

A subset of $A_\mu(G,X,\nu)$ is $A(G,X,\nu)$, the set of measure
preserving actions, for which $a(g)_*\nu = \nu$ for all $g \in
G$. Note that both the weak and the very weak topology on
$A_\mu(G,X,\nu)$ coincide on $A(G,X,\nu)$, and in particular coincide
with the topology studied by Kechris~\cite{Kechris-global-aspects}, Kerr
and Pichot~\cite{kerr2008asymptotic}, and others.

\section{The weak topology on $A_\mu(G,X,\nu)$}

In this section we consider $A_\mu(G,X,\nu)$ with the weak topology. We
will prove Theorem~\ref{thm:T-erg-meager}
and~\ref{thm:non-T-non-dense}.  To this end, the following proposition
will be helpful.

\begin{prop}[Kaimanovich and Vershik~\cite{kaimanovich1983random}]
  \label{prop:bounded-rn}
  Fix a generating probability measure $\mu$ on $G$.  Then for each $g \in
  G$ there exist constants $M_g,N_g > 0 $ such that for every $a \in
  A_\mu(G,X,\nu)$ and $\nu$-almost-every $x \in X$ 
  \begin{align*}
    -N_g \leq \log \frac{da(g)_*\nu}{d\nu}(x) \leq M_g,
  \end{align*}
  and furthermore
  \begin{align*}
    \sum_g\mu(g)M_g < H(\mu),
  \end{align*}
  where $H(\mu)$ is the Shannon entropy of $\mu$.
\end{prop}

\subsection{Property (T) groups}
In this section we prove Theorem~\ref{thm:T-erg-meager}.  We will need
a few definitions and lemmas.

Let $G \cc (X,\nu)$ be a nonsingular ergodic action on a standard
probability space, and, as above, let $G$ act on $L^2(X,\nu)$ by
$$(gf)(x)=f(g^{-1}x) \sqrt{\frac{dg_*\nu}{d\nu}(x)}$$
for $g\in G, f\in L^2(X,\nu), x\in X$. 
\begin{lem}
  If $\nu$ is not equivalent to a $G$-invariant finite measure, then
  there does not exist a nonzero $G$-invariant vector in $L^2(X,\nu)$.
\end{lem}
\begin{proof}
  To obtain a contradiction, suppose $f\in L^2(X,\nu)$ is a nonzero
  $G$-invariant vector. So
  $$f(x)=f(g^{-1}x) \sqrt{\frac{dg_*\nu}{d\nu}(x)}$$
  for a.e. $x\in X$ and every $g\in G$. Because $f$ is nontrivial, the
  support $S=\{x\in X:~f(x)\ne 0\}$ is a $G$-invariant subset which is
  not a null set. Because the action is ergodic, $\nu(S)=1$.

  Define the measure $\eta$ on $X$ by $d\eta(x) = f(x)^2d\nu(x)$. Then
  $$d\eta(g^{-1}x) = f(g^{-1}x)^2~d\nu(g^{-1}x) = f(g^{-1}x)^2 \frac{dg_*\nu}{d\nu}(x) ~d\nu(x) = f(x)^2~d\nu(x) =  d\eta(x).$$
  Thus $\eta$ is a $G$-invariant measure equivalent to $\nu$, and
  moreover has finite mass $\|f\|^2_2$, a contradiction.
\end{proof}
\begin{cor}
  \label{cor:non-invariant-vecs}
  Let $G \cc (X,\nu)$ be an ergodic $\mu$-stationary action on a standard
  probability space. If $\nu$ is not invariant then there does not
  exist a nonzero $G$-invariant vector in $L^2(X,\nu)$.
\end{cor}
\begin{proof}
  It is proven in~\cite[Proposition 2.6]{bader2006factor} that any
  non-invariant $\mu$-stationary action on a standard probability
  space cannot be equivalent to an invariant action. Hence the
  corollary follows from the previous lemma.
\end{proof}

Denote the ergodic actions in $A_\mu(G,X,\nu)$ by $A^e_\mu(G,X,\nu)$.
Let $A'$ be the set of all actions $a \in A_\mu(G,X,\nu)$ such that
there exists a subset $Y \subset X$ with the following properties:
\begin{enumerate}
\item $0 < \nu(Y) < 1$,
\item $a(g)(Y)=Y$ for all $g \in G$,
\item $a$ restricted to $Y$ is $\nu$-measure-preserving.
\end{enumerate}

\begin{lem}
  \label{lem:erg-closure}
  If $G$ has property (T) then the closure of $A^e_\mu(G,X,\nu)$ is
  disjoint from $A'$.
\end{lem}
\begin{proof}
  Choose $a \in A'$, and let $Y$ have the properties listed above.  In
  this case, the vector $\ind_Y$ is an invariant vector in
  $L^2(X,\nu)$ for $a$. We will prove the lemma by showing that $a$
  cannot be a limit of ergodic actions.

  Assume the contrary, so that $a_n \to_n a$ in the weak topology,
  with each $a_n \in A^e_\mu(G,X,\nu)$. Let $G \curvearrowright
  (X^\N,\eta)$ be an ergodic stationary joining (see,
  e.g.,~\cite{furstenberg2009stationary}) of $\{a_n \curvearrowright
  (X,\nu)\}_{n \in \N}$; the action on the $n$\textsuperscript{th}
  coordinate is by $a_n$. Since each $a_n$ is ergodic, the projection
  of $\eta$ on the $n$\textsuperscript{th} coordinate is $\nu$. Define
  $f_n \colon X^\N \to \R$ by
  $$f_n(x)=\ind_Y(x_n)$$
  where $x_n\in X$ is the $n$-th coordinate projection of $x$. Then,
  \begin{align*}
    \|gf_n- f_n\| = \|a_n(g)\ind_Y - \ind_Y\| =\|a_n(g)\ind_Y - a(g)\ind_Y\|,
  \end{align*}
  where the first equality follows from the definition of $f_n$, and
  the second follows from the second and third properties of $A'$. Note also that $\|f_n\|=\nu(Y)^{1/2}$. 

  It follows from the definition of the weak topology that
  \begin{align*}
    \lim_n \|gf_n- f_n\| = 0,
  \end{align*}
  and $\{f_n\}_{n\in\N}$ is a sequence of almost-invariant vectors in
  $L^2(X^\N,\eta)$. Since $G$ has property (T), there must be a nonzero
  invariant vector in $L^2(X^\N,\eta)$. By
  Corollary~\ref{cor:non-invariant-vecs}, $G \cc (X^\N,\eta)$ must be
  measure-preserving. This implies each action $a_n$ is
  measure-preserving and therefore the limit action $a$ is also
  measure-preserving. By~\cite[Theorem 1]{glasner1997kazhdan} (see also~\cite[Theorem 12.2]{Kechris-global-aspects}), because $G$ has property (T), the set of ergodic measure-preserving actions is closed in the space of all measure-preserving actions. Since each $a_n$ is ergodic this implies $a$ must also be ergodic, a contradiction. 
    \end{proof}

\begin{proof}[Proof of Theorem~\ref{thm:T-erg-meager}]
  We will prove the theorem by showing that $A'$ is dense in
  $A_\mu(G,X,\nu)$. It will then follow from
  Lemma~\ref{lem:erg-closure} that $A^e_\mu(G,X,\nu)$ is nowhere
  dense.

  Let $a$ be any action in $A_\mu(G,X,\nu)$. Without loss of
  generality, we may assume that $X = [0,1]$ and $\nu$ is the Lebesgue
  measure. Let $T_n \colon \R \to \R$ be the linear map
  $T_n(x)=(1-1/n)\cdot x$. Define $a_n \in A'$ by
  \begin{align*}
    a_n(g)(x) =
    \begin{cases}
      [T_n \circ a(g) \circ T_n^{-1}](x)& \mbox{for } x \in [0,1-1/n]\\
      x& \mbox{for } x \in (1-1/n,1].
    \end{cases}
  \end{align*}
  Note that $a_n$ is still stationary, and indeed in $A'$; the set
  $(1-1/n,1]$ is a non-trivial invariant set. We will show that $a_n
  \to a$ by showing that for every $g \in G$ and measurable $A
  \subseteq [0,1]$ it holds that
  \begin{align*}
    \lim_n \|a_n(g)\ind_A - a(g)\ind_A\| = 0.
  \end{align*}
   Fix $g \in G$, and denote
  \begin{align}
    \label{eq:def-r}
    r(x) = \frac{da(g)_*\nu}{d\nu}(x), \quad   r_n(x) = \frac{da_n(g)_*\nu}{d\nu}(x).
  \end{align}

  For any interval $A \subseteq [0,1]$ it holds that
  \begin{align*}
    \langle a_n(g)\ind_A, a(g)\ind_A \rangle =
    \int_{a_n(g)A \cap a(g)A}\sqrt{r_n(x) \cdot r(x)}d\nu(x).
  \end{align*}
  Now, on $[0,1-1/n]$, by the definition of $a_n$ it holds that
  $r_n(x) = r(T_n^{-1}x)$.  On $(1-1/n,1]$, $r_n(x)=1$ and by
  Proposition~\ref{prop:bounded-rn}, there exists a constant $C =
  C(\mu,g)$ such that $\sqrt{ r(x)} < C$. Hence
  \begin{align*}
     \langle a_n(g)\ind_A, a(g)\ind_A \rangle = e_n +
    \int_{a_n(g)A \cap a(g)A} r(x)d\nu(x) = e_n+[a(g)_*\nu](a_n(g)A \cap a(g)A),
  \end{align*}
  where the error term $e_n$ satisfies 
 \begin{eqnarray*}
 |e_n| &\le& C/n + \int_0^{1-1/n} \left|\sqrt{r(T_n^{-1}x)r(x)} - r(x)\right| ~dx \\
 &\le& C/n + C \int_0^{1/1-n} \left|\sqrt{r(T_n^{-1}x)} - \sqrt{r(x)}\right| ~dx.
 \end{eqnarray*}

We will show that this error terms tends to zero as $n\to\infty$. For this, let $\epsilon>0$ and let $f$ be a continuous function on $[0,1]$ such that $\|f-\sqrt{r}\|_1<\epsilon$. Then
\begin{eqnarray*}
&&\int_0^{1-1/n} \left|\sqrt{r(T_n^{-1}x)} - \sqrt{r(x)}\right| ~dx\\
&\le& \int_0^{1-1/n} \left|\sqrt{r(T_n^{-1}x)} - f(T_n^{-1}x) \right| +\left| f(T_n^{-1}x) - f(x) \right|+  \left|f(x) - \sqrt{r(x)}  \right| ~dx\\
  &\le& \epsilon + \int_0^{1-1/n} \left| f(T_n^{-1}x) - f(x) \right| ~dx + \epsilon.
  \end{eqnarray*}
  The middle term tends to $0$ as $n\to\infty$ by the Bounded Convergence Theorem. Since $\epsilon$ is arbitrary, this implies $\lim_{n\to\infty} |e_n| = 0$.

  Thus
  \begin{align*}
    \lim_n  \langle a_n(g)\ind_A, a(g)\ind_A \rangle =
    \lim_n [a(g)_*\nu](a_n(g)A \cap a(g)A).
  \end{align*}
  Since $a_n(g)$ is measurable and $\nu$-nonsingular, $\nu(a_n(g)A
  \vartriangle a(g)A)$ tends to zero. To be more precise, we observe that 
  \begin{eqnarray*}
  \nu(a_n(g)A  \vartriangle a(g)A) &\le& 1/n + \nu\big(a_n(g)(A \cap [0,1-1/n])  \vartriangle a(g)A\big) \\
  &=& 1/n + \nu\big(T_n a(g)T_n^{-1}(A \cap [0,1-1/n])  \vartriangle a(g)A\big) \\
    &\le& 1/n +\nu\big(a(g)T_n^{-1}(A \cap [0,1-1/n])  \vartriangle T_n^{-1}a(g)A\big).
    \end{eqnarray*}
  So it suffices to show that if $B \subset [0,1-1/n]$ is any measurable set then $\lim_{n\to\infty} \nu( B \vartriangle T_n^{-1} B) = 0$. This follows by approximating the characteristic function $1_B$ by a continuous function. Thus
 \begin{eqnarray*}
      \lim_n  \langle a_n(g)\ind_A, a(g)\ind_A \rangle &=& \lim_n [a(g)_*\nu](a_n(g)A \cap a(g)A) \\
      &=& \lim_n [a(g)_*\nu](a_n(g)A \cup a(g)A) -[a(g)_*\nu](a_n(g)A \vartriangle a(g)A) \\
      &=& \lim_n [a(g)_*\nu](a_n(g)A \cup a(g)A) \ge \lim_n [a(g)_*\nu](a(g)A) = \nu(A).
      \end{eqnarray*}
      Since  $[a(g)_*\nu](a_n(g)A \cap a(g)A) \le [a(g)_*\nu](a(g)A) = \nu(A)$, we must have
         \begin{align*}
    \lim_n  \langle a_n(g)\ind_A, a(g)\ind_A \rangle  = \nu(A),
  \end{align*}
  and it follows immediately that
    \begin{align*}
    \lim_n \|a_n(g)\ind_A - a(g)\ind_A\| = 0.
  \end{align*}
\end{proof}

\subsection{Entropy gaps and non-property (T) groups}
Recall that $A^e_\mu(G,X,\nu)$ is the set of ergodic actions in
$A_\mu(G,X,\nu)$.  Following Theorem~\ref{thm:T-erg-meager}, and in
light of what is known about measure preserving actions, it is natural
to ask if $A^e_\mu(G,X,\nu)$ is residual when $G$ does not have
property (T).  The purpose of this section is to give a negative
answer to this question.

In particular, we will show that there exist groups $G$ without
property (T) and appropriately chosen $\mu$ for which
$A^e_\mu(G,X,\nu)$ is not dense in $A_\mu(G,X,\nu)$, when the latter
is equipped with the weak topology.

Recall that the {\em Furstenberg entropy} of an action $a \in
A_\mu(G,X,\nu)$ is given by
\begin{align*}
  h_\mu(a) = \sum_{g \in
    G}\mu(g)\int_X-\log\left(\frac{d\nu}{da(g)_*\nu}(x)\right)~da(g)_*\nu(x).
\end{align*}

By Nevo~\cite{nevo2003spectral}, if $G$ has property (T) then any
generating measure $\mu$ has an {\em entropy gap}.  Namely, there
exists a constant $\epsilon=\epsilon(\mu)$ such that the Furstenberg
entropy of any non-measure preserving $\mu$-stationary ergodic action
$a$ is at least $\epsilon$.

It is known that some non property (T) groups do not have an entropy
gap (e.g. free groups~\cite{bowen2010random}, some lamplighter groups,
and $SL_2(\Z)$~\cite{hartman2012furstenberg}).  However, we will next
describe groups with an entropy gap which fail to have property (T).

Let $\mu = \mu_1 \times \mu_2$ be a product of generating measures on
a product group $G_1 \times G_2$ (i.e., the support of each $\mu_i$
generates $G_i$ as a semigroup), and let $G \cc (X,\nu)$ be
$\mu$-stationary. Let $(X_i,\nu_i) = G_i \slashslash X$ be the space
of $G_i$-ergodic components of $(X,\nu)$ (i.e., the Mackey realization
of the $G_i$-invariant sigma-algebra), and let $\pi_i \colon X \to
X_i$ be the associated factor map such that $\pi_{i*}\nu =
\nu_i$. Note that the $G$-action on $(X_i,\nu_i)$ factors through
$G_{3-i}$ for $i=1,2$ (that is, it has $G_i$ in its kernel). Then
by~\cite[Proposition~1.10]{bader2006factor}, the map $\pi = \pi_1
\times \pi_2 \colon X \to X_1 \times X_2$ pushes $\nu$ to $\nu_1
\times \nu_2$, and is furthermore relatively measure preserving. It
follows that
\begin{enumerate}
\item $G \cc (X_1 \times X_2, \nu_1 \times \nu_2)$ is an ergodic,
  $\mu$-stationary action, $G_1 \cc (X_2,\nu_2)$ is an ergodic,
  $\mu_1$-stationary action, and likewise $G_2 \cc (X_1,\nu_1)$ is an
  ergodic, $\mu_2$-stationary action.
\item $h_\mu(X,\nu) = h_{\mu_1}(X_2,\nu_2)+h_{\mu_2}(X_1,\nu_1)$.
\end{enumerate}
Given this, we are ready to state and prove the following claim.
\begin{prop}
  \label{prop:non-T-gap}
  There exists a group $G$ that does not have property (T), and a
  generating probability measure $\mu$ on $G$ (which may be taken to
  have finite entropy) such that $\mu$ has an entropy gap.
\end{prop}
\begin{proof}
  Let $G = G_1 \times G_2 = \Gamma \times \Z$, where $\Gamma$ has
  property (T). Let $\mu_1$ be a generating measure on $\Gamma$,
  $\mu_2$ a generating measure on $\Z$, and let $\mu=\mu_1 \times
  \mu_2$. Note that $\mu$ can be taken to have finite entropy.

  Let $a$ be an ergodic, $\mu$-stationary action on $(X,\nu)$, and let
  $(X_1 \times X_2, \nu_1 \times \nu_2)$ be given as above. Denote by
  $a_i$ the induced action of $G_i$ on $(X_{3-i},\nu_{3-i})$. Then
  $h_\mu(a) = h_{\mu_1}(a_1)+h_{\mu_2}(a_2)$.

  Now, since $G_2=\Z$ is abelian, $h_{\mu_2}(a_2)=0$. Since $G_1$ has
  property (T), then $\mu_1$ has an entropy gap, and so, since $G_1
  \cc (X_2,\nu_2)$ is ergodic, $h_{\mu_1}(a_1)$ is either zero or
  larger than some $\epsilon$ that depends only on $\mu$, and the
  proof is complete.
\end{proof}
Note that one can replace $(\Z,\mu_2)$ with any amenable group $G_2$
and a measure $\mu_2'$ such that $(G_2,\mu_2')$ has a trivial Poisson
boundary.

\begin{prop}\label{prop:gap-not-dense}
  If $(G,\mu)$ has an entropy gap then $A^e_\mu(G,X,\nu)$ is not dense
  in $A_\mu(G,X,\nu)$.
\end{prop}

In order to prove the the proposition, we first observe the following. 
 
\begin{lem}
  \label{lem:entropy-cont}
  The Furstenberg entropy is a continuous map 
  $h \colon A_\mu(G,X,\nu) \to \R$ where $A_\mu(G,X,\nu)$ is equipped with the weak topology.
\end{lem}
\begin{proof}
  Define
  \begin{align*}
    h_\mu^g(a) = \int_X-\log\left(\frac{d\nu}{da(g)_*\nu}(x)\right)~da(g)_*\nu(x),
  \end{align*}
  so that
  \begin{align*}
    h_\mu(a) = \sum_g\mu(g)h_\mu^g(a).
  \end{align*}
  It follows from Proposition~\ref{prop:bounded-rn} that the maps
  $h_\mu^g \colon A_\mu(G,X,\nu) \to \R^+$ are each bounded by a
  constant $M_g$, such that $\sum_g\mu(g)M_g < \infty$. We will hence
  prove the claim by showing that each of the maps $h_\mu^g$ is
  continuous, and applying the Bounded Convergence Theorem.

  Let $\{a_n\} \subset A_\mu(G,X,\nu)$ converge to $a \in
  A_\mu(G,X,\nu)$. Fix $g \in G$ and define $r(x)$ and $r_n(x)$ as
  in~\eqref{eq:def-r}.  Then
  \begin{align*}
    \lim_n \int_X|r_n(x)-r(x)| ~d\nu(x) = 0.
  \end{align*}
  By Proposition~\ref{prop:bounded-rn},  $-N_g \leq \log r_n(x) \leq
  M_g$, and likewise $-N_g \leq \log r_n(x) \leq
  M_g$. Hence
  \begin{align*}
    \lim_n \int_X|\log r_n(x)- \log r(x)| ~d\nu(x) = 0,
  \end{align*}
  and since
  \begin{align*}
    h_\mu^g(a_n) = \int_X\log r_n(x)~d\nu(x)\quad\mbox{ and }\quad h_\mu^g(a) = \int_X\log r(x)~d\nu(x)
  \end{align*}
  we have shown that $\lim_nh_\mu^g(a_n) = h_\mu^g(a)$.
\end{proof}

\begin{proof}[Proof of Proposition~\ref{prop:gap-not-dense}]
  Let $\mu$ be a generating measure on $G$ with finite Shannon entropy and an entropy gap. Choose $t\in [0,h_\mu(\Pi(G,\mu))]$, where $\Pi(G,\mu)$ is the action on
  the Poisson boundary.  By weighting properly the disjoint union of a
  non-atomic measure preserving action and the action on the Poisson
  boundary, we can construct a non-ergodic $\mu$-stationary action $a$
  with entropy $t$.

  Assume that the ergodic actions are dense. Then there exists a
  sequence of ergodic actions $a_n \to a$, and by
  Lemma~\ref{lem:entropy-cont}, $h_\mu(a_n) \to h_\mu(a)$. This means
  that the entropy values can be realized by ergodic stationary
  actions is dense in $[0,h_\mu(\Pi(G,\mu))]$, which cannot be the case
  when $\mu$ has an entropy gap.
\end{proof}

The combination of Propositions~\ref{prop:non-T-gap}
and~\ref{prop:gap-not-dense} yields our desired result,
Theorem~\ref{thm:non-T-non-dense}. Incidentally, it is now easy to prove Proposition \ref{prop:no-01-law}:

\begin{proof}[Proof of Proposition \ref{prop:no-01-law}]
  Let $t$ be any real number with $0<t< h_\mu(\Pi(G,\mu))$ where, as
  above, $\Pi(G,\mu)$ denotes the Poisson boundary of $(G,\mu)$. Let
  $K \subset A_\mu(G,X,\nu)$ be the set of all actions $a$ with $0\le
  h_\mu(a) \le t$. Then $K$ is $\Aut(X,\nu)$-invariant. Because
  entropy is continuous by Lemma \ref{lem:entropy-cont}, $K$ is
  closed. However, since it does not contain any action
  measurably-conjugate to the Poisson boundary, it is not dense. The
  closure of the complement of $K$ is also not dense since it does not
  contain any measure-preserving actions (since these have entropy
  0). Thus $K$ is neither residual nor meager, and also a dense
  $\Aut(X,\nu)$-orbit cannot exist in $A_\mu(G,X,\nu)$.
\end{proof}

\section{A correspondence principle}\label{sec:correspondence}

In this section, we endow $A_\mu(G,X,\nu)$ with the very weak
topology, and prove Theorem~\ref{thm:generic-equivalence}: a general
correspondence between generic dynamical properties of $\cP_\mu(X^G)$
and those of $A_\mu(G,X,\nu)$. This generalizes a result of Glasner
and King \cite{GK98} from the measure-preserving case to the
stationary case and from the circle to an arbitrary perfect Polish
space.  We begin by studying the topology of the group $\Aut^*(X,\nu)$
of nonsingular transformations of a Lebesgue probability space
$(X,\nu)$ (\S~\ref{sec:aut}) from which we constructed a topology on
$A_\mu(G,X,\nu)$ (\S~\ref{sec:action-space}).


\subsection{Dynamical generic-equivalence}

The group $\Aut(X,\nu)$ acts on $A^*(G,X,\nu)$ by conjugations:
$$(Ta)(g)=Ta(g)T^{-1},\quad T \in \Aut(X,\nu), a \in A^*(G,X,\nu).$$
This action is by homeomorphisms. The orbit of $a$ under this action
is its {\em measure-conjugacy class}. More generally, if $G \cc
(X',\nu')$ and $G \cc (X'',\nu'')$ are nonsingular actions and if
there exists a $G$-equivariant measurable isomorphism $\phi:X' \to
X''$ (ignoring sets of measure zero) such that $\phi_*\nu' = \nu''$ then
we say these two actions are {\em measurably conjugate} and write $G
\cc (X',\nu') \sim G \cc (X'',\nu'')$

Suppose $\Omega$ is a topological space and for each $\omega \in
\Omega$ there is assigned a nonsingular action $G \cc
(X_\omega,\nu_\omega)$. Then $(\Omega,\{ G \cc
(X_\omega,\nu_\omega)\}_{\omega \in \Omega})$ is a {\em
  setting}~\cite{GK98}. For example, $A^*(G,X,\nu)$ is a setting. On
the other hand, suppose $G \curvearrowright Y$ is a jointly continuous
topological action of a countable group on a compact Hausdorff
space. Recall that $\cP^*_G(Y)$ is the set of Borel probability
measures $\eta \in \cP(Y)$ such that the action $G \cc (Y, \eta)$ is
nonsingular. Then $\cP^*_G(Y)$ is a setting: we associate to each
$\eta \in \cP^*_G(Y)$ the system $G \cc (Y,\eta)$.

Similarly, $A_\mu(G,X,\nu)$ and $\cP_\mu(Y)$ are
settings. Following~\cite{GK98} we are interested in comparing the
dynamical properties of settings. To be precise, suppose $(\Omega,\{ G
\cc (X_\omega,\nu_\omega)\}_{\omega \in \Omega})$ is a setting. A subset
$P \subset \Omega$ is a {\em dynamical property} if for every
$\omega_1,\omega_2 \in \Omega$
$$G \cc (X_{\omega_1},\nu_{\omega_1}) \sim G \cc (X_{\omega_2},\nu_{\omega_2}) \Leftrightarrow (\{\omega_1,\omega_2\} \subset P \textrm{ or } \{\omega_1,\omega_2\} \cap P = \emptyset).$$
Observe that if $\Omega'$ is another setting and $P \subset \Omega$ is
dynamical then there exists a corresponding dynamical set $P' \subset
\Omega'$: it is the set of all $\omega' \in \Omega'$ such
that there exists $\omega \in P$ with
$$G \cc (X_{\omega},\nu_{\omega}) \sim G \cc ({X'}_{\omega'},{\nu'}_{\omega'}).$$

A subset $P \subset \Omega$ is {\em Baire} if it can be written as the
symmetric difference of $O$ and $M$, where $O \subset \Omega$ is open
and $M \subset \Omega$ is meager. The Baire sets form a
sigma-algebra which contains the Borel sigma algebra.

Two settings $\Omega_1,\Omega_2$ are {\em dynamically
  generically-equivalent} if for every dynamical property $P_1 \subset
\Omega_1$ if $P_2 \subset \Omega_2$ is the corresponding property then
$P_1$ is Baire/residual/meager in $\Omega_1$ iff $P_2$ is
Baire/residual/meager in $\Omega_2$.


The following is a formal rephrasing of
Theorem~\ref{thm:generic-equivalence}. It is the main result of this section.
\begin{thm}
  \label{thm:dyn-gen-equivalent}
  Let $(G,\mu)$ be a countable group with a generating probability
  measure, $Z$ be a perfect Polish space and let $G \cc Z^G$ denote
  the shift action. Let $(X,\nu)$ be standard, non-atomic probability
  space. Then $A_\mu(G,X,\nu)$ and $\cP_\mu(Z^G)$ are dynamically
  generically-equivalent.
\end{thm}

\begin{remark}
  One can replace $\cP_\mu(Z^G)$ with $\cP_G(Z^G)$ and
  $A_\mu(G,X,\nu)$ with $A(G,X,\nu)$ in the proof of
  Theorem~\ref{thm:generic-equivalence}. So the proof shows that
  $A(G,X,\nu)$ and $\cP_G(Z^G)$ are also dynamically
  generically-equivalent. This extends the Glasner-King
  Theorem~\cite{GK98} which shows that $A(G,X,\mu)$ and $\cP_G(\T^G)$
  are dynamically generically-equivalent where $\T=S^1$ represents the
  1-dimensional torus. Our proof is based on~\cite{GK98}. We need a
  few additional arguments to generalize from $\T$ to $Z$ and we fill
  a few gaps in the somewhat terse presentation in~\cite{GK98}.
\end{remark}

\label{sec:dyn-gen-equi}


\subsection{Preliminaries and outline of the proof of
  Theorem~\ref{thm:dyn-gen-equivalent}}

Let 
\begin{itemize}
\item $\bB=\N^\N$ denote the Baire space,
\item $\bcP(\bB)\subset \cP(\bB)$ denote the subspace of
  fully-supported purely non-atomic Borel probability measures on
  $\bB$,
\item $\bcP_\mu(\bB^G) \subset \cP_\mu(\bB^G)$ the subspace of
  $\mu$-stationary probability measures whose projections on each
  coordinate are in $\bcP(\bB)$. Note that it is enough to require
  that the projection to the identity coordinate is in $\bcP(\bB)$.
\end{itemize}

The first step in the proof of Theorem \ref{thm:generic-equivalence}
is to prove:
\begin{prop}
  \label{claim:residual}
  $\bcP_\mu(\bB^G)$ and $\cP_\mu(\Pol^G)$ are dynamically generically
  equivalent.
\end{prop}
This result is essentially due to the fact that any perfect Polish
space contains a dense $G_\delta$-subset homeomorphic to $\bB$. It is
proven in \S~\ref{sec:Baire}.

Because of Proposition~\ref{claim:residual}, it suffices to prove that
$\bcP_\mu(\bB^G)$ and $A_\mu(G,X,\nu)$ are dynamically
generically-equivalent. Without loss of generality, we may assume that
$(X,\nu) = (\bB,\lambda)$, where $\lambda$ is any fully-supported
purely non-atomic Borel probability measure on $\bB$.

Our proof proceeds as follows. We construct a Polish space $\H$ and
consider the setting $\H \times A_\mu(G,\bB,\lambda)$ (that depends on
the second coordinate only). We show
(Proposition~\ref{prop:E-properties}) that there exists a map $E
\colon \H \times A_\mu(G,\bB,\lambda) \to \bcP_\mu(\bB^G)$ satisfying:
\begin{enumerate}
\item for any $a \in A_\mu(G,\bB,\lambda)$ and $h\in \H$, the action
  $G \cc (\bB^G,E(h,a))$ is measurably conjugate to $a$.
\item The image of $E$ is residual in $\bcP_\mu(\bB^G)$.
\item $E$ is a homeomorphism onto its image.
\end{enumerate}

Given this, the proof of Theorem~\ref{thm:dyn-gen-equivalent} is
straight-forward.
\begin{proof}[Proof of Theorem~\ref{thm:dyn-gen-equivalent}]
  Let $P_1 \subseteq A_\mu(G,\bB,\lambda)$ be a dynamical property,
  and let $P_2 \subseteq \bcP_\mu(\bB^G)$ and $P_3 \subseteq
  \cP_\mu(\Pol^G)$ be the corresponding properties. We would like to
  show that $P_1$ is $(*)$ iff $P_3$ is $(*)$, where $(*)$ stands for
  either Baire, residual or meager. By Proposition~\ref{claim:residual},
  $P_2$ is $(*)$ iff $P_3$ is $(*)$.

  By the first property of $E$, $P_2 = E(P_1 \times \H) \cup M$, for
  some $M$ in the complement of the image of $E$. By the second
  property of $E$, $M$ is meager. Hence, by the third property of $E$,
  $P_2$ is $(*)$ iff $\H \times P_1$ is $(*)$. Finally, by~\cite[Page
  57]{Ox71}, $P_1 \times \H$ is $(*)$ iff $P_1$ is $(*)$, and so the
  claim follows.
\end{proof}



\subsection{Reduction to Baire space}\label{sec:Baire}
The purpose of this subsection is to prove Proposition \ref{claim:residual}.

\subsubsection{The non-atomic, fully supported measures are a residual subset}

A well known fact is that every perfect Polish space has a dense
$G_\delta$ subset that is homeomorphic to the Baire space $\bB =
\N^\N$ (see, e.g., the proof of Proposition 2.1
in~\cite{burke2003models}). Denote by $\hat \Pol$ such a subset of
$\Pol$. Recall that $\bcP_\mu(\hat\Pol^G)$ denotes the set of
$\mu$-stationary measures on $\hat\Pol^G$ (a dense $G_\delta$ subset
of $\Pol^G$ that is homeomorphic to $\bB^G$), which furthermore have
marginal on the identity coordinate that is fully supported and
non-atomic.

To prove Proposition~\ref{claim:residual}, we need the following
lemma.  Before stating it, we note that by the portmanteau Theorem, if
$Y$ is a subset of $X$ then the space of all probability measures on
$X$ that are supported on $Y$ is homeomorphic with $\cP(Y)$. In
particular, we think of $\bcP_\mu(\hat\Pol^G)$ as a subset of
$\cP_\mu(\Pol^G)$.
\begin{lem}
  \label{clm:nu-k}
  $\bcP_\mu(\hat\Pol^G)$ is dense in $\nu \in \cP_\mu(\Pol^G)$. 
\end{lem}

\begin{proof}
  Given a measure $\nu \in \cP_\mu(\Pol^G)$, we construct a sequence
  $\{\nu_n\} \subset \bcP_\mu(\hat\Pol^G)$ such that $\lim_n \nu_n =
  \nu$. 
  
  Fix a non-atomic, fully supported $\lambda \in \cP(\hat\Pol)$.  Fix
  a compatible metric on $\Pol$, and denote by $B(x,r)$ the ball of
  radius $r$ around $x \in \Pol$.  For $x \in \Pol$ and $n \in \N$,
  let $\lambda_{x,n} \in \cP(\Pol)$ be equal to $\lambda$, conditioned
  on $B(x,1/n)$. That is, for any measurable $A \subseteq \Pol$, let
  \begin{align*}
    \lambda_{x,n}(A) = \frac{\lambda(A \cap B(x,1/n))}{\lambda(B(x,1/n))}.
  \end{align*}
  This is well defined, since $\lambda$ is fully supported and so
  $\lambda(B(x,1/n)) > 0$. For $\xi \in \Pol^G$, let $\lambda_n^\xi =
  \prod_{g \in G}\lambda_{\xi(g),n} \in \cP(\hat\Pol^G)$ be the
  product measure with marginal $\lambda_{\xi(g),n}$ on coordinate
  $g$. Note that $g_*\lambda_n^\xi = \lambda_n^{g\xi}$. 
  
    Let $\bar\nu_n \in \cP(\hat\Pol^G)$ be given by
  \begin{align*}
    \bar\nu_n = \int_{\Pol}\lambda_n^\xi d\nu(\xi).
  \end{align*}
  Then, since $g_*\lambda_n^\xi = \lambda_n^{g\xi}$, 
  \begin{align*}
    \sum_{g \in G}\mu(g)g_*\bar\nu_n = \int_{\Pol}\lambda_n^\xi
    d\left(\sum_{h \in G}\mu(g)dg_*\nu\right)(\xi),
  \end{align*}
  which, by the $\mu$-stationarity of $\nu$, is equal to $\bar\nu_n$. Hence
  $\bar\nu_n \in \cP_\mu(\hat\Pol^G)$. Note that $\bar\nu_n$ is
  non-atomic, since each $\lambda_n^\xi$ is non-atomic.

  Let
  \begin{align*}
    \nu_n = \frac{1}{n}\lambda^G+\frac{n-1}{n}\bar\nu_n.
  \end{align*}
  Since $\lambda^G$ is invariant, $\nu_n$ is
  $\mu$-stationary. Furthermore, $\ell_*\nu_n \in \bcP(\hat\Pol)$
  where $\ell:Z^G\to Z$ is the projection map to the identity
  coordinate. Hence $\nu_n \in \bcP_\mu(\hat\Pol^G)$.

  It remains to be shown that $\lim_n\nu_n = \nu$. Clearly, this will
  follow if we show that $\lim_n\bar\nu_n = \nu$, since $\lambda^G/n$
  converges to the zero measure. Let $d(\cdot,\cdot)$ be the
  compatible metric on $\Pol$ used to define $\lambda_n^\xi$, let $G =
  \{g_1,g_2,\ldots\}$ and let
  \begin{align*}
    \hat d(\xi,\xi') = \sum_i 2^{-i}d(\xi(g_i),\xi'(g_i))
  \end{align*}
  be a compatible metric on $\Pol^G$. Then the support of
  $\lambda_n^\xi$ is contained in a $\hat d$-ball of radius $1/n$
  around $\xi$. Hence $\lim_n\lambda^\xi_n = \delta_\xi$. It follows
  by the bounded convergence theorem that
  \begin{align*}
    \lim_{n\to\infty}\bar\nu_n =
    \lim_{n\to\infty}\int_{\Pol}\lambda_n^\xi d\nu(\xi)
    = \int_{\Pol}\delta_\xi d\nu(\xi)
    = \nu.
  \end{align*}

\end{proof}

\begin{proof}[Proof of Proposition~\ref{claim:residual}]
  Let $\beta \colon \bB \to \hat\Pol$ be a homeomorphism from the
  Baire space to the dense $G_\delta$ subset $\hat\Pol \subset
  \Pol$. We extend $\beta$ to a map $\bB^G \to \hat\Pol^G$ by acting
  independently in each coordinate. Thus $\beta$ is a homeomorphism
  between $\bB^G$ and $\hat\Pol^G$, which, furthermore, commutes with
  the $G$-action. Hence $\beta_*$ is a homeomorphic measure conjugacy
  between $\cP_\mu(\bB^G)$ and $\cP_\mu(\hat\Pol^G)$ and thus
  $\bcP_\mu(\bB^G)$ and $\bcP_\mu(\hat\Pol^G)$ are dynamically
  generically equivalent. Clearly, the natural embedding
  $\bcP_\mu(\hat\Pol^G) \hookrightarrow \cP_\mu(\Pol^G)$ is a homeomorphic
  measure conjugacy.  Accordingly, we prove the claim by showing that
  $\bcP_\mu(\hat\Pol^G)$ is residual in $\cP_\mu(\Pol^G)$.  
  
  By Lemma~\ref{clm:nu-k}, $\bcP_\mu(\hat\Pol^G)$ is dense in
  $\cP_\mu(\Pol^G)$.  It thus remains to be shown that
  $\bcP_\mu(\hat\Pol^G)$ is a $G_\delta$. We do this in two
  steps. First, we show that $\bcP_\mu(\hat\Pol^G)$ is $G_\delta$ in
  $\cP_\mu(\hat\Pol^G)$. Then, we show that $\cP_\mu(\hat\Pol^G)$ is a
  $G_\delta$ in $\cP_\mu(\Pol^G)$.

  Identifying $\hat\Pol$ with $\bB$, we show that $\bcP_\mu(\bB^G)$ is
  $G_\delta$ in $\cP_\mu(\bB^G)$. Let $\cS$ be a countable base 
  for the topology of $\bB$.  For
  $S \in \cS$, the set
  \begin{align*}
    U_{S,n} = \{\nu \in \cP_\mu(\bB^G)\,:\,\ell_*\nu(S)>1/n\} 
  \end{align*}
  is open by the portmanteau Theorem, where, to remind the reader,
  $\ell_*\nu$ is the projection of $\nu$ on the identity
  coordinate. Let
  \begin{align*}
    F =\bigcap_{S \in \cS} \bigcup_{n \in \N}U_{S,n}
  \end{align*}
  be the measures with an identity marginal that is supported
  everywhere on $\bB$. Note $F$ is a $G_\delta$.

  Having an atom of mass at least $1/n$ is a closed property. Hence
  the complementary set
  \begin{align*}
    W_n = \{\nu \in \cP_\mu(\bB^G)\,:\,\ell_*\nu(\{x\}) < 1/n \mbox{
      for all } x \in \bB\} 
  \end{align*}
  is open. Let $N = \bigcap_nW_n$ denote the measures with non-atomic marginals. Hence
  \begin{align*}
    \bcP_\mu(\bB^G) = F \cap N
  \end{align*}
  is a $G_\delta$ in $\cP_\mu(\bB^G)$. It remains to be shown
  that $\cP_\mu(\hat\Pol^G)$ is a $G_\delta$ in $\cP_\mu(\Pol^G)$.

  Let $\Pol^G \setminus \hat\Pol^G = \cup_{k \in \N}C_k$, where $C_k
  \subset \Pol^G$ is closed. These exist since $\hat\Pol^G$ is a
  $G_\delta$ in $\Pol^G$. Then
  \begin{align*}
    V_{k,n} = \{\nu \in \cP_\mu(\Pol^G) \,:\, \nu(C_k) < 1/n\} 
  \end{align*}
  is open by the portmanteau Theorem. Hence
  \begin{align*}
    \cP_\mu(\hat\Pol^G) = \bigcap_{k \in \N}\bigcap_{n \in \N}V_{k,n}
  \end{align*}
  is a $G_\delta$ in $\cP_\mu(\Pol^G)$.
\end{proof}

\subsection{The Polish group $\H$}

We now proceed to construct $\H$. We defer some of the proofs to
Appendix~\ref{app:proofs-dyn-gen}.

Denote by $\bcP([0,1])$ the space of fully supported, non-atomic
measures on the interval $[0,1]$, endowed with the weak* topology on
$\cP([0,1])$.  Given $\nu \in \bcP([0,1])$, we can define its
``inverse cumulative distribution function'' $h \colon [0,1] \to
[0,1]$ by
\begin{align*}
  h^{-1}(t) = \nu([0,t)).
\end{align*}
It is easy to verify that $h$ is an order preserving homeomorphism of
the interval $[0,1]$ that fixes $0$ and $1$. Let $\H$ denote the group
of all such homeomorphisms, with the topology inherited from the space
of continuous functions on $[0,1]$. Then $\H$ is
Polish~\cite{GK98}. Since $h_* \nu \in \bcP([0,1])$ for any $h \in \H$
and $\nu \in \bcP([0,1])$, $\H$ acts on $\bcP([0,1])$. This action is
jointly continuous~\cite{GK98}, free and transitive.

Denote by $\Irr$ the space of irrational numbers in the interval
$[0,1]$, equipped with the subspace topology inherited from the
interval. Because of continued fractions expansions this space is homeomorphic to the Baire space $\bB$. Since
$\Irr$ differs from $[0,1]$ by a countable number of points,
$\bcP([0,1])$ and $\bcP(\Irr)$ can be identified. Let $\alpha \colon \bB
\to \Irr$ be a homeomorphism.  Then $\H$ acts on $\bcP(\bB)$ by
\begin{align*}
  h\nu := \alpha_*^{-1}h_*\alpha_*\nu.
\end{align*}
Since $\H$ acts continuously, transitively and freely on $\bcP(\Irr)$,
it also acts continuously, transitively and freely on $\bcP(\bB)$.

Denote by $\bcP(\bB^G)$ (resp.\ $\bcP([0,1]^G)$) the space of
probability measures whose projections on each coordinate are in
$\bcP(\bB)$ (resp.\ $\bcP([0,1])$).

Extending $\alpha$ to a map $\bB^G \to \Irr^G$ by acting independently
in each coordinate, it follows that $\alpha_* \colon \bcP(\bB^G) \to
\bcP([0,1]^G)$ is a homeomorphism. Hence, as above, we can define a
continuous action of $\H$ on $\bcP(\bB^G)$ by
\begin{align*}
  h\nu = \alpha_*^{-1}h_*\alpha_*\nu,
\end{align*}
where here the action $\H \curvearrowright \Irr^G$ is also independent
on each coordinate. As before, the action $\H \curvearrowright
\bcP(\bB^G)$ is continuous. Since $\alpha$ acts on each coordinate
separately, the $\H$-action on $\bcP(\bB^G)$ commutes with the $G$-action by
shifts.

Let $\ell \colon \bB^G \to \bB$ be the projection on the identity
coordinate, given by $\ell(\xi) = \xi(e)$.  Denote by $\bcP_\mu(\bB^G)
= \cP_\mu(\bB^G) \cap \bcP(\bB^G)$ the space of $\mu$-stationary
measures whose projection on each coordinate is in
$\bcP(\bB)$. Equivalently, one can just require that the projection on
the identity coordinate be in $\bcP(\bB)$, since the different
marginals measures are mutually absolutely continuous.

Note that the $\H$ action on $\bcP(\bB^G)$ restricts to an action on
$\bcP_\mu(\bB^G)$, since the $\H$- and $G$-actions on $\bcP(\bB^G)$
commute.  This concludes our definition of $\H$ as a Polish group
acting on $\bcP_\mu(\bB^G)$.

It thus remains to be shown that $A_\mu(G,\bB,\lambda)$ and
$\bcP_\mu(\bB)$ are dynamically generically equivalent. 

\subsection{The map $E$}\label{sec:E}

For $\zeta \in \cP(\bB)$, denote by
\begin{align*}
  \cP_\mu^\zeta(\bB^G) = \{\nu \in
  \cP_\mu(\bB^G)\,:\,\ell_*\nu = \zeta\}
\end{align*}
the set of $\mu$-stationary measures on $\bB^G$ with marginal $\zeta$
on the identity coordinate. Recall that $\lambda$ is any
fully-supported purely non-atomic Borel probability measure on $\bB$;
we will be interested in $\cP^\lambda_\mu(\bB^G)$. Note that
$\cP^\lambda_\mu(\bB^G) \subseteq \bcP_\mu(\bB^G)$, since $\lambda$ is
non-atomic and supported everywhere.

Define
\begin{equation*}
  \begin{array}{rcrcl}
    D&\colon&\H \times \cP^\lambda_\mu(\bB^G) &\longrightarrow &\bcP_\mu(\bB^G)\\
    & & (h,\nu) &\longmapsto     &h_*\nu.
\end{array}
\end{equation*}

In~\cite{GK98} (see specifically Equation 14 and the preceding
remark) it is shown that the map $\H \times \cP^\zeta_\mu([0,1]^G) \to
\bcP_\mu([0,1]^G)$ given by $(h,\nu) \mapsto h_*\nu$ is a
homeomorphism, for $\zeta$ the Lebesgue measure on $[0,1]$. Since
$\bcP_\mu([0,1]^G)$ and $\bcP_\mu(\bB^G)$ are homeomorphic\footnote{As
  are $\cP^\zeta_\mu([0,1]^G)$ and $\cP^\lambda_\mu(\bB^G)$; one can
  take $\lambda = \alpha^{-1}_*\zeta$.}, it follows that
\begin{lem}
  \label{lem:D}
  $D$ is a homeomorphism.
\end{lem}

Given $a \in A_\mu(G,\bB,\lambda)$, let $\pi_a \colon \bB \to \bB^G$
be given by
\begin{align}
  \label{eq:pi}
  [\pi_a(x)](g) = a(g^{-1})x. 
\end{align}
For a fixed $a$, $\pi_a$ is $G$-equivariant, since for all $k,g \in G$
and $x \in \bB$,
\begin{align*}
  [\pi_a(a(k)x)](g) = a(g^{-1})a(k)x = a(g^{-1}k)x =
  [\pi_a(x)](k^{-1}g) = [k\pi_a(x)](g).
\end{align*}
It follows that $\pi_{a*}\lambda \in \cP_\mu(\bB^G)$. Furthermore,
$\ell \circ \pi_a$ is the identity, and so in particular
$\pi_{a*}\lambda \in \cP^\lambda_\mu(\bB^G)$.  Define
\begin{equation*}
  \begin{array}{rcrcl}
    F&\colon&A_\mu(G,\bB,\lambda) &\longrightarrow &\cP^\lambda_\mu(\bB^G)\\
    & & a&\longmapsto     &\pi_{a*}\lambda.
\end{array}
\end{equation*}

$F$ is one to one, since disintegrating $\pi_{a*}\lambda$ with respect
to the projection $f \mapsto f(e)$ from $\bB^G$ to $\bB$ yields point
mass distributions as the fiber measures, from which the $a$-orbits of
$\lambda$-a.e.\ $x \in \bB$ can be reconstructed.  Also, $a$ and
$F(a)$ are always measurable conjugate, with $\pi_a$ being the
$G$-equivariant measurable isomorphism.  We furthermore prove in
Appendix~\ref{app:proofs-dyn-gen} the following two lemmas.  Analogues
of these lemmas appear in~\cite{GK98} (see pages 239 and 240), for the
measure preserving setting.
\begin{lem}
  \label{lem:F}
  $F$ is homeomorphism onto its image.
\end{lem}
\begin{lem}
  \label{lem:F-res}
  The image of $F$ is residual in $\cP^\lambda_\mu(\bB^G)$.
\end{lem}
We prove these lemmas in the appendix. 

Define
\begin{equation*}
  \begin{array}{rcrcl}
    E&\colon&\H \times A_\mu(G,\bB,\lambda) &\longrightarrow &\bcP_\mu(\bB^G)\\
    & & (h, a)&\longmapsto     &h_*\pi_{a*}\lambda.
\end{array}
\end{equation*}
Alternatively, $E(h,a) = D(h, F(a))$.

The following proposition establishes the properties of $E$ needed for
Theorem~\ref{thm:dyn-gen-equivalent}.
\begin{prop}
  \label{prop:E-properties}
  $E$ has the following properties:
  \begin{enumerate}
  \item $a \in A_\mu(G,\bB,\lambda)$ and $G\cc (\bB^G,E(h,a))$ are measurably
    conjugate for all $h \in \H$.
  \item $E$ is a homeomorphism onto its image.
  \item The image of $E$ is residual in $\bcP_\mu(\bB^G)$.
  \end{enumerate}
\end{prop}
\begin{proof}
  \begin{enumerate}
  \item This follows from the fact that both $\pi_a$ and $\H$ commute
    with $G$; given an action $a \in A_\mu(G,\bB,\lambda)$, the
    equivariant isomorphism (up to $\lambda$-null sets) between
    $\lambda$ and $E(h,a)$ is given by $h \circ \pi_a \colon \bB \to
    \bB^G$.
  \item Since $E(h,a) = D(h, F(a))$, and since $D$ and $F$ are both
    homeomorphisms onto their images (Lemmas~\ref{lem:D}
    and~\ref{lem:F}), it follows that $E$ is a homeomorphism onto its
    image.
  \item This can be seen by considering the following sequence of
    embeddings, each of which - as we explain below - is a
    homeomorphic embedding with a residual image:
    \begin{align*}
      \H \times A_\mu(G,\bB,\lambda) \quad \overset{\mathrm{id} \times
        F}{\longrightarrow} \quad \H \times \cP^\lambda_\mu(\bB^G)
      \quad \overset{D}{\longrightarrow} \quad \bcP_\mu(\bB^G).
    \end{align*}
    
    By Lemmas~\ref{lem:F-res} and~\ref{lem:F}, $F$ embeds $A_\mu(G,\bB,\lambda)$
    homeomorphically into a residual subset of
    $\cP^\lambda_\mu(\bB^G)$. Hence $\mathrm{id} \times F$ embeds $\H
    \times A_\mu(G,\bB,\lambda)$ into a residual subset of $\H \times
    \cP^\lambda_\mu(\bB^G)$, by~\cite[Page 57]{Ox71}.

    Since $D$ is a homeomorphism between $\H \times
    \cP^\lambda_\mu(\bB^G)$ and $\bcP_\mu(\bB^G)$ (Lemma~\ref{lem:D}),
    it follows that $E = D \circ (\mathrm{id} \times F)$ embeds $\H
    \times A_\mu(G,\bB^G)$ into a residual subset of $\bcP_\mu(\bB^G)$.

  \end{enumerate}
\end{proof}

\section{Applications of the correspondence principle} 
In this section we prove Theorems~\ref{thm:generic-poisson-boundary}
and~\ref{thm:weak-rohlin}.

\begin{proof}[Proof of Theorem~\ref{thm:generic-poisson-boundary}]
  Let $\X = \{0,1\}^\N$ be the Cantor space, equipped with the usual
  product topology. Let $\cP^{ext}_\mu(\X^G) \subset \cP_\mu(\X^G)$ be
  the subset of all measures $\eta$ such that $G \cc (\X^G,\eta)$ is
  an essentially free ergodic extension of the Poisson boundary. By
  Theorem~\ref{thm:poisson-dense} $\cP^{ext}_\mu(\X^G)$ is dense in
  $\cP_\mu(\X^G)$. It is well-known that an action is an extension of
  the Poisson boundary if and only if it has maximal $\mu$-entropy. So
  if $H(\mu)<\infty$ then by Theorem~\ref{thm:G-deltas},
  $\cP^{ext}_\mu(\X^G)$ is a $G_\delta$ subset of
  $\cP_\mu(\X^G)$. Since a dense $G_\delta$ is residual, and since
  $\X$ is a perfect Polish space, it follows from
  Theorem~\ref{thm:dyn-gen-equivalent} that the same holds for
  $A_\mu(G,X,\nu)$ (with the very weak topology). By another
  application of Theorem~\ref{thm:dyn-gen-equivalent}, the same also
  holds for $\cP_\mu(\Pol^G)$, where $\Pol$ is any perfect Polish
  space.

  The second claim of this theorem follows in a similar way.
\end{proof}

\begin{proof}[Proof of Theorem \ref{thm:weak-rohlin}]
  Recall from the proof of Theorem \ref{thm:generic-poisson-boundary}
  above that $\cP^{ext}_\mu(\X^G)$ is a residual subset of
  $\cP_\mu(\X^G)$.  So the Correspondence Principle (Theorem
  \ref{thm:generic-equivalence}) implies that the subset
  $A^{ext}_\mu(G,X,\nu)$ of all actions $a\in A_\mu(G,X,\nu)$ that are
  ergodic essentially free extensions of the Poisson boundary is
  residual in $A_\mu(G,X,\nu)$. 
  
  Let $b \in A_\mu(G,X,\nu)$. We will show $b$ is in the closure of
  $\Aut(X,\nu)a$. By \S \ref{sec:vweak-space}, it suffices to show
  that for every $\epsilon>0$, measurable partition
  $\cP=\{P_1,\ldots,P_n\}$ of $X$ and finite $W \subset G$ there
  exists $a' \in \Aut(X,\nu)a$ such that
  $$\sup_{1\le i,j\le n}\sup_{g\in W}  | \nu(b(g)P_i \cap P_j) - \nu(a'(g)P_i \cap P_j)|<\epsilon.$$
  Let $\phi:X \to \{1,\ldots, n\}$ be the map $\phi(x)=i$ if $x \in P_i$. Let $\Phi:X \to \{1,\ldots, n\}^G$ be the map $\Phi(x)_g = \phi(b(g)^{-1}x)$. Observe that this is $G$-equivariant with respect to the $b$-action. Let 
  $$Y_i=\{y \in \{1,\ldots, n\}^G:~ y_e = i\}.$$
   By Theorem \ref{thm:0-1-dense}, there exists a $\mu$-stationary probability measure $\kappa$ on $\{1,\ldots, n\}^G$ such that
  \begin{itemize}
  \item $G \cc (\{1,\ldots,n\}^G,\kappa)$ is a $G$-factor of $G \cc^a (X,\nu)$;
  \item $ \sup_{g\in W} \sum_{i,j=1}^n |\Phi_*\nu(Y_i \cap gY_j) - \kappa(Y_i \cap gY_j)|<\epsilon/n^2$.
  \end{itemize}
  Let $\Psi:X \to \{1,\ldots,n\}^G$ be a $G$-factor of $a$ so that $\kappa = \Psi_*\nu$. Let $Q'_i = \Psi^{-1}(Y_i)$.  Observe that $\cQ'=\{Q'_1,\ldots, Q'_n\}$ is a measurable partition of $X$ and 
  $$ \sup_{g\in W}\sum_{i,j=1}^n | \nu(b(g)P_i \cap P_j) - \nu(a(g)Q'_i \cap Q'_j)|<\epsilon/n^2.$$
Because $\cP$ is a partition, the equation above implies the existence of a measurable partition $\cQ=\{Q_1,\ldots, Q_n\}$ of $X$ such that
\begin{itemize}
\item $\nu(P_i)=\nu(Q_i)$ for all $i$
\item $\nu(Q_i \vartriangle Q'_i) < \epsilon/n$ for all $i$.
\end{itemize}

Let $\psi \in \Aut(X,\nu)$ be any measure-preserving transformation such that $\psi(Q_i)=P_i$ for all $i$. Define $a' \in A_\mu(G,X,\nu)$ by $a'(g) =\psi a(g)\psi^{-1}$. It follows that
 $$ \sup_{1\le i,j\le n}\sup_{g\in W} | \nu(b(g)P_i \cap P_j) - \nu(a'(g)P_i \cap P_j)|<\epsilon.$$
Thus $\Aut(X,\nu)a$ is dense as required.
\end{proof}

\begin{proof}[Proof of Corollary \ref{cor:0-1law}]
  In order to deduce the 0-1 law, recall the 0-1 Lemma
  from~\cite{GK98}: Let $\A$ be a BaireCat space and let $\Phi$ be a
  group of homeomorphisms of $\A$, such that there exists a $T \in \A$
  with a dense $\Phi$-orbit. Then each Baire-measurable
  $\Phi$-invariant subset of $\A$ is either residual or meager.

  The term ``BaireCat space'' refers to a space that satisfies the
  Baire Category Theorem. In particular, Polish spaces are BaireCat.
  Since $\Aut(X,\nu)$ acts continuously on the Polish space
  $A_\mu(G,X,\nu)$, the corollary follows from
  Theorem~\ref{thm:weak-rohlin}.
\end{proof}
\appendix


\section{The weak* topology on a space of measures defined by a
  relative property}
\label{sec:weak}

The purpose of this section is to define the weak* topology on a space
of measures defined by a relative property. To be precise, let
$(V,\nu)$ be a standard Borel probability space and $W$ a  Polish space. Let $\cP(V\times W|\nu)$ denote the set of all Borel
probability measures on $V\times W$ that project to $\nu$. We will
show that several natural topologies on this set are equal.

For this purpose, let us assume that $V$ is also a Polish space. The
weak* topology on $\cP(V\times W)$ is defined by: a sequence
$\{\lambda_n\}_{n=1}^\infty$ converges to $\lambda_\infty$ if and only
if: for every compactly supported continuous function $f$ on $V\times
W$, $\int f~d\lambda_n$ converges to $\int f~d\lambda_\infty$. We
regard $\cP(V\times W|\nu)$ as a subspace of $\cP(V\times W)$. We will
show that the subspace topology on $\cP(V\times W| \nu)$ does not
depend on the choice of Polish structure for $V$. This justifies the
following definition: The {\em weak* topology} on $\cP(V\times W|\nu)$ is
the subspace topology inherited from the inclusion $\cP(V\times W|\nu)
\subset \cP(V\times W)$ where $V$ is endowed with an arbitrary Polish
structure and $\cP(V\times W)$ with the usual weak* topology.

Let $\MALG(\nu)$ denote the measure algebra of $\nu$. To be precise,
$\MALG(\nu)$ consists of all measurable subsets of $V$ modulo null
sets. For $A, B \in \MALG(\nu)$ we let $d(A,B)=\nu(A \vartriangle
B)$. With this metric, $\MALG(\nu)$ is a complete separable metric
space.

Let $\cM(W)$ denote the space of all finite Borel measures on $W$ with
the weak* topology. Let $\Map(\MALG(\nu),\cM(W))$ denote the space of
all maps from $\MALG(\nu)$ to $\cM(W)$ with the pointwise convergence
topology. This space has a natural convex structure as it may be
identified with the product space $\cM(W)^{\MALG(\nu)}$.

Given a measure $\lambda \in \cP(V\times W|\nu)$ let $v\mapsto
\lambda^v$ be a measurable map from $V$ to $\cP(W)$ such that
$$\lambda = \int \delta_v \times \lambda^v~d\nu(v).$$
It is a standard fact that such a map exists and is unique up to null sets. 

Define $$\Phi:\cP(V\times W| \nu) \to \Map(\MALG(\nu),\cM(W))$$ by
$$\Phi(\lambda)(A)=\lambda^A$$
where $\lambda^A = \int_A \lambda^v~d\nu(v)$.
\begin{prop}\label{prop:weak}
  The map $\Phi$ is an affine homeomorphism onto its image.
\end{prop}

\begin{remark}
  The topology on $\Map(\MALG(\nu),\cM(W))$ is independent of the
  topology on $V$. So this proposition shows that the topology on
  $\cP(V\times W| \nu)$ is independent of the topology on $V$.
\end{remark}

\begin{proof}
  Any Borel measure $\lambda$ on $V\times W$ is determined by its
  values on sets of the form $A\times B$ where
  $A\subset V, B \subset W$ are Borel. Note that
  $$\lambda(A\times B) = \lambda^A(B) = \Phi(\lambda)(A)(B).$$
  This proves that $\Phi$ is injective.

  Suppose $\lambda_n \in \cP(V\times W| \nu)$ and $\lim_n \lambda_n =
  \lambda_\infty$ in the weak* topology. To prove that $\Phi$ is
  continuous, it suffices to show that $\lambda_n^A \to
  \lambda^A_\infty$ for every $A \in \MALG(\nu)$. Actually, it
  suffices to show that this is true for every $A$ in a dense subset
  of $\MALG(\nu)$ (because $\lambda$ is completely determined by the
  values $\lambda^A$ for $A$ in a dense subset of $\MALG(\nu)$).
  
  Recall that if $Z$ is any topological space and $Y \subset Z$ then $\partial Y = \overline{Y} \cap
  \overline{Z\setminus Y}$. Given a measure $\zeta$ on $Z$ we say $Y$ is a {\em continuity set} of
  $\nu$ if $\nu(\partial Y)=0$. It is not difficult to show that the
  collection of all continuity sets forms a dense subalgebra of
  $\MALG(\zeta)$ if $Z$ is a Polish space (see, e.g.,~\cite[Lemma 8.4]{bowen2012entropy}; the
  important requirement is the regularity of $\zeta$). 
  
  The portmanteau
  Theorem implies that $\lim_n \lambda_n(E) = \lambda_\infty(E)$ for
  any set $E \subset V\times W$ which is a continuity set for
  $\lambda_\infty$. In particular, $\lim_n \lambda_n^A(B) =
  \lambda_\infty^A(B)$ if $A$ is a continuity set for $\nu$ and $B$ is
  a continuity set for the projection of $\lambda^\infty$ to
  $W$. Because continuity sets of $\Proj_W(\lambda_\infty)$ are dense
  in $\MALG(\Proj_W(\lambda_\infty))$, this implies that $\lim_n
  \lambda_n^A = \lambda_\infty^A$ in the weak* topology on $\cP(W)$
  for every continuity set $A$ of $\nu$. Because continuity sets of
  $\nu$ are dense in $\MALG(\nu)$, it follows that $\Phi(\lambda_n)
  \to \Phi(\lambda_\infty)$ as $n\to\infty$. Because $\lambda$ is
  arbitrary, $\Phi$ is continuous.

  It is clear that $\Phi$ is affine. If $V$ is compact then
  $\Phi^{-1}$ must be continuous on the image of $\Phi$. This proves
  the proposition when $V$ is compact.

  Suppose $V$ is non-compact and let $\{\lambda_n\} \subset
  \cP(V\times W|\nu), \lambda_\infty \in \cP(V\times W|\nu)$ be
  measures such that $\Phi(\lambda_n)$ converges to
  $\Phi(\lambda_\infty)$ as $n\to\infty$. If $K \subset V$ is a
  compact subset then the considerations above imply that $\lambda_n
  \res K \times W\to \lambda_\infty \res K \times W$. Thus if $f \in
  C_c(V\times W)$ is a compactly supported continuous function then
  $$\lim_n \int f~d\lambda_n = \int f~d\lambda_\infty.$$
  This implies $\lambda_n$ converges to $\lambda_\infty$ in the weak*
  topology. So the inverse of $\Phi$ is also continuous which implies
  the proposition.


\end{proof}

\begin{cor}\label{cor:weak}
  Let $(\lambda_n) \subset \cP(V\times W|\nu)$ be a sequence of
  measures. Then the following are equivalent
  \begin{enumerate}
  \item $\lambda_n$ converges to a measure $\lambda_\infty$ in the
    weak* topology on $\cP(V\times W|\nu)$ with respect to any (every)
    Polish structure on $V$.

  \item $\lambda_n^A \to \lambda_\infty^A$ for every $A \in
    \MALG(\nu)$.
  \item $\lambda_n^A \to \lambda_\infty^A$ for every $A$ in some dense
    subset of $\MALG(\nu)$.
  \end{enumerate}
  Moreover, if $\lambda_n^v \to \lambda_\infty^v$ for a.e. $v\in V$
  then (1-3) above hold.
\end{cor}

\begin{proof}
  It follows immediately from the previous proposition that the first
  three items are equivalent. Suppose $\lambda_n^v \to
  \lambda_\infty^v$ for a.e. $v\in V$. Let $V$ be endowed with a
  Polish topology. Let $f \in C_c(V\times W)$. Then
  \begin{eqnarray*}
    \lim_n \int f ~d\lambda_n &=& \lim_n \iint f(v,w) ~d\lambda^v_n(w) d\nu(v) = \iint  \lim_n f(v,w) ~d\lambda^v_n(w) d\nu(v)\\
    &=& \iint f(v,w) ~d\lambda^v_\infty(w) d\nu(v) = \int f ~d\lambda_\infty
  \end{eqnarray*}
  by the Bounded Convergence Theorem. Since $f$ is arbitrary, this
  shows $\lambda_n$ converges to $\lambda_\infty$ in the weak*
  topology.
\end{proof}

\section{Proofs of Lemmas from \S~\ref{sec:E}} \label{app:proofs-dyn-gen}

\subsection{A clopen base}\label{sec:more-defs}
Let $\cS_1, \cS_2, \ldots$ be a sequence of finite partitions of
$\bB$, such that each $S \in \cS_n$ is a clopen set, $\cS_{n+1}$ is a
refinement of $\cS_n$, and $\cS = \cup_n\cS_n$ is a countable base of
the topology of $\bB$, and also of the associated Borel
sigma-algebra.

Likewise, let $\cT_1, \cT_2, \ldots$ be a sequence of finite
partitions of $\bB^G$ with the same properties, and likewise denote
$\cT = \cup_n\cT_n$. Furthermore, let each $T \in \cT$ be of the form
\begin{align*}
  T = \bigcap_{g \in K}g\ell^{-1}(S_g) = \bigcap_{g \in K}\{\xi \in
  \bB^G\,:\,\xi(g) \in S_g\}.
\end{align*}
for some finite $K \subset G$ and a map $g \mapsto S_g \in \cS_n$, for
some $n \in \N$.  This is, again, a countable base of the topology and
of the Borel sigma-algebra.

\subsection{$F$ is a homeomorphism onto its image}

\begin{proof}[Proof of Lemma~\ref{lem:F}]
  It is shown in \S \ref{sec:E} that  $F \colon A_\mu(G,\bB,\lambda) \to
  \cP^\lambda_\mu(\bB^G)$ is injective.  Hence it remains to be shown
  that it is continuous and that its inverse is continuous.

  Recall (see \S~\ref{sec:more-defs}) that $\cS$ is a base for the sigma-algebra of $\bB$.  For $S_1,S_2 \in \cS$, define
  \begin{align*}
    A_{S_1,S_2,g} = \{\xi \in \bB^G\,:\, \xi(e) \in S_1,\xi(g) \in S_2\}.
  \end{align*}
  Let $\zeta = F(a)$. By the definition of $F$,
  \begin{align}
    \label{eq:nu-lambda}
    \zeta(A_{S_1,S_2,g}) = \lambda\left(S_1\cap a(g)S_2\right).
  \end{align}

  Let $\{a_n\}$ be a sequence in $A_\mu(g,\bB,\lambda)$, and denote
  $\nu_n=F(a_n)$.  Since each $\nu_n$ is a graph measure, it is
  determined by the two-dimensional marginals, or equivalently by the
  values of $\nu_n(A_{S_1,S_2,g})$. Hence $\lim_n\nu_n = \nu$ iff
  \begin{align}
    \label{eq:nu-n}
    \lim_n\nu_n(A_{S_1,S_2,g}) = \nu(A_{S_1,S_2,g})
  \end{align}
  for all $S_1,S_2 \in \cS$ and $g \in G$.

  We first show that $F^{-1}$ is continuous by showing that if
  $\lim_n\nu_n = \nu$
  then $\lim_na_n = a$, where $a = F^{-1}(\nu)$.
  Assume then that~\eqref{eq:nu-n} holds. Combining it
  with~\eqref{eq:nu-lambda} yields
  
  \begin{align}  \label{eq:lim-a-2}
  \lim_n \lambda(S_1 \cap a_n(g) S_2) = \lambda(S_1 \cap a(g) S_2),
\end{align}
  which implies $\lim_n a_n= a$. So $F^{-1}$
  is continuous.

  Analogously, to see that $F$ is continuous, assume
  that~\eqref{eq:lim-a-2} holds. Combining this
  with~\eqref{eq:nu-lambda} yields   
  \begin{align*}
    \lim_n\nu_n(A_{S_1,S_2,g}) = \nu(A_{S_1,S_2,g}).
  \end{align*}

\end{proof}

\subsection{The image of $F$ is residual}
To prove Claim~\ref{lem:F-res}, we show that $\Im F$ is a dense
$G_\delta$ in $\cP_\mu^\lambda(\bB^G)$; in Claim~\ref{clm:F-dense}
we show that it is dense, and in Claim~\ref{clm:F-G-delta} we show
that it is a $G_\delta$.

\begin{claim}
  \label{clm:F-dense}
  $\Im F$ is dense in $\cP_\mu^\lambda(\bB^G)$.
\end{claim}
\begin{proof}
  Recall (see \S~\ref{sec:more-defs}) that $\cS$ is a clopen base of
  the topology of $\bB$, and $\cT$ is a clopen base of
  the topology of $\bB^G$.
  
  Given $\nu \in \cP_\mu^\lambda(\bB^G)$, we construct for each $n
  \in \N$ an action $a_n \in A_\mu(G,\bB,\lambda)$ such that for all
  $T \in \cT$ it holds that $[F(a_n)](T) = \nu(T)$ for $n$ large
  enough. This will prove the claim, since it implies that $\lim_n
  F(a_n) = \nu$.

  Fix $n$.  Then $\{\ell^{-1}(S)\}_{S \in \cS_n}$ is a finite
  partition of $\bB^G$ where $\ell:\bB^G\to\bB$ is the projection map to the identity coordinate.  Denote by $\nu|_{\ell^{-1}(S)}$ the measure $\nu$
  restricted to $\ell^{-1}(S)$, and define $\lambda|_S$ analogously.
  
  Let $\varphi_n \colon \bB^G \to \bB$ be a measurable map that, for
  each $S \in \cS_n$, is a measure isomorphism between
  $(\ell^{-1}(S),\nu|_{\ell^{-1}(S)})$ and $(S,\lambda|_S)$.  This is possible,
  since all of the spaces $(\ell^{-1}(S),\nu|_{\ell^{-1}(S)})$ and
  $(S,\lambda|_S)$ are standard non-atomic finite measure
  spaces, with the same total mass.  It follows that $\varphi_n$ is a
  measure isomorphism between $(\bB^G,\nu)$ and $(\bB,\lambda)$.

  Now, let $a_n \in A_\mu(G,\bB,\lambda)$ be given by, for every $g \in
  G$ and $x \in \bB$,
  \begin{align}
    \label{eq:a_n-def}
    [a_n(g)](x) = \varphi_ng\varphi_n^{-1}x,
  \end{align}
  where the $G$-action is here by shifts on $\bB^G$. This is indeed a
  $\mu$-stationary action, since it is conjugate to $(\bB^G,\nu)$.

  Finally, for $T \in \cT$, we show that $[F(a_n)](T) = \nu(T)$ for
  $n$ large enough. By the definition of $\cT$, for $n \in \N$ large
  enough there exists a finite $K \subset G$ and a map $g \mapsto S_g$
  from $K \to \cS_n$ such that
  \begin{align}
    \label{eq:T-decompo}
    T = \bigcap_{g \in K}g\ell^{-1}(S_g) = \bigcap_{g \in K}\{\xi
    \in \bB^G\,:\,\xi(g) \in S_g\}.
  \end{align}

  By the definitions of $F$ and $\pi_{a_n}$,
  \begin{align*}
    [F(a_n)](T) = [\pi_{a_n*}\lambda](T) 
    = \lambda(\{x \in \bB\,:\,\pi_{a_n}(x) \in T\}).
  \end{align*}
  By~\eqref{eq:T-decompo}
  \begin{align*}
    = \lambda(\{x \in \bB\,:\,[\pi_{a_n}(x)](g) \in S_g \mbox{ for all } g
    \in K\}).
  \end{align*}
  Hence by the definitions of $\pi_{a_n}$ and $a_n$
  (in~\eqref{eq:a_n-def})
  \begin{align*}
    &= \lambda(\{x \in \bB\,:\,a_n(g^{-1})(x) \in S_g
    \mbox{ for all } g \in K\})\\
    &= \lambda(\{x \in \bB\,:\,\varphi_ng^{-1}\varphi_n^{-1}x \in S_g
    \mbox{ for all } g \in K\}).
  \end{align*}
  Now, $\varphi_n^{-1}S_g = \ell^{-1}(S_g)$, and so
  \begin{align*}
    &= \lambda(\{x \in \bB\,:\,\varphi_n^{-1}x \in
    g\ell^{-1}(S_g) \mbox{ for all } g \in K\}).
  \end{align*}
  But $\varphi_{n*}^{-1}\lambda = \nu$ and so
  \begin{align*}
    = \nu(\{\xi \in \bB^G\,:\,\xi \in g\ell^{-1}(S_g) \mbox{ for all } g \in K\})
    = \nu(T).
  \end{align*}

  Thus $\Im F$ is dense in $\cP^\lambda_\mu(\bB^G)$.
\end{proof}

\begin{claim}
  \label{clm:F-G-delta}
  $\Im F$ is a $G_\delta$ subset of $\cP_\mu^\lambda(\bB^G)$.
\end{claim}
\begin{proof}
  A classical result states that if a subset of a metric space is
  completely metrizable then it is a $G_\delta$ (see,
  e.g.,~\cite[Theorem 12.3]{Ox71}). Since $\cP_\mu^\lambda(\bB^G)$ is
  Polish, and since its subset $\Im F$ is the homeomorphic image
  (Lemma~\ref{lem:F}) of the Polish space $A_\mu(G,\bB,\lambda)$, it
  follows that $\Im F$ is a $G_\delta$ subset of
  $\cP_\mu^\lambda(\bB^G)$.
\end{proof}






\bibliography{poulsen}
\bibliographystyle{abbrv}

\end{document}